\documentclass[12pt,reqno]{amsart}
\usepackage[utf8]{inputenc}
\usepackage{amssymb, amsmath, amsthm, url}
\usepackage{hyperref}
\usepackage{todonotes}
\usepackage{mathptmx}
\usepackage{enumerate}
\usepackage[shortlabels]{enumitem}

\def\sideremark#1{\ifvmode\leavevmode\fi\vadjust{\vbox to0pt{\vss % the remark
      \hbox to 0pt{\hskip\hsize\hskip1em           %                will appear only
 \vbox{\hsize2cm\tiny\raggedright\pretolerance10000%                on the side
 \noindent #1\hfill}\hss}\vbox to8pt{\vfil}\vss}}} %
\usepackage{pgf,tikz,pgfplots}
\pgfplotsset{compat=1.12}
\usepackage{mathrsfs}
\usetikzlibrary{arrows}

\newtheorem{introtheorem}{Theorem}

\newtheorem{introconjecture}[introtheorem]{Conjecture}
\newtheorem{theorem}{Theorem}[section]
\newtheorem{lemma}[theorem]{Lemma}

\newtheorem{proposition}[theorem]{Proposition}
\newtheorem{corollary}[theorem]{Corollary}

\theoremstyle{definition}
\newtheorem{definition}[theorem]{Definition}

\theoremstyle{remark}
\newtheorem{remark}[theorem]{Remark}

\usepackage{color}

\newcommand{\cc}{\ensuremath{C^*}}
\newcommand{\dd}{\ensuremath{D^*}}

\newcommand{\muu}{\ensuremath{\mu (D^*, E_r)}}

\newcommand{\ord}{\ensuremath{\operatorname {ord}}}

\numberwithin{figure}{section}

\begin{document}

\title[On the valuative Nagata conjecture]{On the valuative Nagata conjecture}

\author[Galindo]{Carlos Galindo}

\address{Universitat Jaume I, Campus de Riu Sec, Departamento de Matem\'aticas \& Institut Universitari de Matem\`atiques i Aplicacions de Castell\'o, 12071
Caste\-ll\'on de la Plana, Spain} \email{galindo@uji.es}

\author[Monserrat]{Francisco Monserrat}

\address{Universitat Polit\`ecnica de Val\`encia, Departamento de Matem\'atica Aplicada \&  Instituto Universitario de Matem\'atica Pura y Aplicada, 46022
Valencia, Spain} \email{framonde@mat.upv.es}

\author[Moreno-\'Avila]{Carlos-Jes\'us Moreno-\'Avila}

\address{Universitat Jaume I, Campus de Riu Sec, Departamento de Matem\'aticas \& Institut Universitari de Matem\`atiques i Aplicacions de Castell\'o, 12071
Caste\-ll\'on de la Plana, Spain} \email{cavila@uji.es}

\author[Moyano-Fern\'andez]{Julio-José Moyano-Fernández}

\address{Universitat Jaume I, Campus de Riu Sec, Departamento de Matem\'aticas \& Institut Universitari de Matem\`atiques i Aplicacions de Castell\'o, 12071
Caste\-ll\'on de la Plana, Spain} \email{moyano@uji.es}

\thanks{The authors were partially funded by MCIN/AEI/10.13039/501100011033 and by ``ERDF A way of making Europe", grants PGC2018-096446-B-C21 and PGC2018-096446-B-C22, as well as by Universitat Jaume I, grant UJI-B2021-02. The third author was also supported by the Margarita Salas postdoctoral contract MGS/2021/14(UP2021-021) financed by the European Union-NextGenerationEU}

\begin{abstract}
We provide several equivalent conditions for a plane divisorial valuation of a smooth projective surface to be minimal with respect to an ample divisor. These conditions involve a valuative Seshadri constant and other global tools of the surface defined by the divisorial valuation. As a consequence, we derive several equivalent statements for the valuative Nagata conjecture and some related results.
\end{abstract}

\maketitle

\section{Introduction}\label{sec:intro}
Nagata in \cite{Nag} gave a negative answer to the 14th problem of Hilbert and, as a consequence, emerged the Nagata conjecture.  Despite many efforts, this conjecture remains open after more than 60 years.  The conjecture is related to several important tools recently used in global geometry of surfaces as Seshadri and Waldschmidt constants or the Mori cone. In fact, there are equivalent statements to the Nagata conjecture involving all those tools.

A valuative Nagata conjecture has been recently stated \cite{DumHarKurRoeSze,GalMonMoy}. This conjecture involves sequences of blowups at infinitely near points on the projective plane defining a (real) plane valuation $\nu$ and it implies the Nagata conjecture. Moreover, the valuative Nagata conjecture  is implied by the Greuel-Lossen-Shustin conjecture \cite[Conjecture 4.7.11]{GreuelBook} and there are asymptotic evidences of its trueness in some particular cases \cite{GalMonMoy}. The main tool used for stating the valuative conjecture is an analogue of the Seshadri constant named Seshadri-type constant, denoted by $\hat{\mu}(\nu)$ and introduced in \cite{BouKurMacSze}. No tool mentioned in the first paragraph has been considered in this context.

The concept of minimal valuation (of the projective plane) is essential  when stating the valuative Nagata conjecture. In this paper we extend this concept to divisorial valuations $\nu$ of smooth projective surfaces $S$ and big divisors $D$ on $S$ (see Section \ref{subsec:mupico}). We provide several statements equivalent to the concept of minimality with respect to ample divisors $D$, and deduce consequences of non-minimality. Our results involve surfaces defined by blowing up a finite chain of infinitely near points on $S$ and, in this setting, a natural valuative Seshadri constant $\epsilon (D, \nu)$ is introduced and related to the value $\hat{\mu}_D(\nu)$ used to define minimality. This allows us to reproduce for these less known surfaces many results one would expect from a valuative Nagata-type conjecture on them. In particular, and when $S$ is the projective plane and $D$ a general line $L$, as our main result, we prove several equivalent statements to the valuative Nagata conjecture which go in parallel to those known in the classical case. It brings to the question of what conditions should we impose on $\nu$ and $D$ so that one gets minimality and their equivalent properties.

The (classical) Nagata conjecture is one of the most stimulating problems for linear systems on the complex projective plane $\mathbb{P}^2$. It predicts the inequality $d > \frac{1}{\sqrt{r}} \sum_{i=1}^r m_i$, where $d$ is the degree of any curve $C$ on $\mathbb{P}^2$ such that ${\rm mult}_{x_i} (C) \geq m_i$, $m_i$ being non-negative integers and $\{x_i\}_{i=1}^r$, $r \geq 10$, very general points in $\mathbb{P}^2$. Nagata proved this result when $r$ is a square and it is an open problem in the remaining cases. The Nagata conjecture has equivalent statements involving several interesting objects used in algebraic geometry \cite{CilHarMirRoe,RoeSup} and it has also extended versions to smooth projective surfaces and nef divisors \cite{StrSze,DeFer1,SyzSze,CilHarMirRoe,CilMirRoe}.

One of these objects are {\it Seshadri (multipoint) constants} on surfaces $S$ with respect to nef divisors $D$, usually denoted $\epsilon(S,D,x_1,x_2 \ldots, x_r)$. Motivated by the Fujita conjecture, Demailly introduced these constants (for a variety and a point) in \cite[Section 6]{Dem} and, although they are not useful in that direction, are an important tool. The literature contains many references to these constants on smooth projective surfaces \cite{EinLaz,Bau,Harb1,Gar,HarbRoe,DumKurMacSze1,FarSzeSzpTut,Han,HanMuk,HanPrav,HanHar,Pok,BauGrimSch,MarPok}. It holds that $\epsilon(S,D,x_1,x_2 \ldots, x_r) \leq \sqrt{D^2/r}$ and, when the bound is attained, one says that $\epsilon(S,D,x_1,x_2 \ldots, x_r)$ is maximal. Otherwise, there exists a submaximal curve, that is a curve $C$ on $S$ going through at least a point $x_i$ such that  $\epsilon(S,D,x_1,x_2 \ldots, x_r)= \frac{D\cdot C}{\sum_{i=1}^{r} \text{mult}_{x_i} C}$ \cite[Proposition 1.1]{BauSze}. Setting $S= \mathbb{P}^2$ and $D=L$, the Nagata conjecture is equivalent to the non-existence of submaximal curves (for very general points $\{x_i\}_{i=1}^r$, $r\geq 10$) and it can be generalized to the Nagata-Biran-Szemberg conjecture, which can be stated as follows: $\epsilon(S,D,x_1,x_2 \ldots, x_r)$ is maximal for $D$ ample, $r$ large enough and $\{x_i\}_{i=1}^r$ very general points in an arbitrary smooth projective surface $S$ \cite[Section 2]{StrSze} and \cite[Section 5.1]{Laz1}.

The Nagata conjecture is also related to the {\it Mori cone} of the surface $\widetilde{\mathbb{P}^2}$ defined by the blowup $\pi: \widetilde{\mathbb{P}^2} \rightarrow \mathbb{P}^2$ of $\mathbb{P}^2$ at $r$ points $x_i$ in $\mathbb{P}^2$, $1 \leq i \leq r$. This cone is an essential tool for the minimal model program and it is far from being known for surfaces \cite{DeFer1,CilHarMirRoe,CilMirRoe}. The Nagata conjecture can be equivalently stated by claiming that, when $\{x_i\}_{i=1}^r$ are very general points and $r\geq 10$, the ray generated by the class of the $\mathbb{R}$-divisor $ \sqrt{r} \pi^* L - \sum_{i=1}^r E_i $ is wonderful, where $\pi^*$ means pull-back and the $E_i$'s denote the exceptional divisors created by $\pi$. Recall that a wonderful ray is an irrational nef ray with vanishing self-intersection.

{\it Newton-Okounkov bodies} were introduced by Okounkov \cite{Oko1} and developed by Lazarsfeld and Musta{\c t}{\u a} \cite{LazMus}, and Kaveh and Khovanskii \cite{KavKho}. They give, among other properties, a systematic procedure for constructing toric degenerations of projective varieties. The recent literature has a good number of papers considering these bodies on surfaces for different purposes \cite{CilFarKurLozRoeShr,KurLoz1,RoeSze,GalMonMor2,GalMonMoyNic2,MoyFerNicRoe} and, as we will see, they play a role in this paper.

As we said previously, the Nagata conjecture has a valuative version, where one considers a very general (real) plane valuation $\nu$ of $\mathbb{P}^2$ (see \cite[Definition 4.2]{GalMonMoy}) and uses the inverse of its normalized volume $[\text{vol}^N(\nu)]^{-1}$ instead of the cardinality of a set of very general points of $\mathbb{P}^2$. The conjecture involves a normalized Seshadri-type constant $\hat{\mu}^N(\nu)$ that satisfies $\hat{\mu}^N(\nu) \geq \sqrt{\frac{1}{[\text{vol}^N(\nu)]^{-1}}}$. When the above bound is attained one says that $\nu$ is minimal, and the {\it valuative Nagata conjecture} states that $\nu$ is minimal when it is a very general valuation  and $[\text{vol}^N(\nu)]^{-1} \geq 9$ (see \cite{GalMonMoy}).

This valuative conjecture was only stated in the above terms and, as far as we know, no extension to other surfaces has been studied. In light of the multiple and interesting statements of the  Nagata conjecture we give, in the forthcoming Conjecture \ref{ConjectureP2}, several parallel statements for the valuative Nagata conjecture. In order to increase the generality of our results, this paper considers a smooth (complex) projective surface $S$, a divisorial valuation $\nu_r$ of $S$ (where $r$ is the number of infinitely near points of the configuration of centers of $\nu_r$)  and an ample divisor $D$ on $S$. In this setting we consider the Seshadri-type constant $\hat{\mu}_D (\nu_r)$ (see \eqref{eq:defmupico}) and, in Definition \ref{def:sdnu},  we introduce  a natural valuative Seshadri constant $\varepsilon(D,\nu_r)$. Lemma \ref{lemm:infimo_seshadri} proves that our definition extends that of one point Seshadri constant and gives rise to an ampleness criterion for divisors on smooth projective surfaces involving divisorial valuations of those surfaces (Theorem \ref{thm_criterion_amplitude}).

The first main result in this paper is the following one (Theorem \ref{thm:triangulo} of the paper), which gives several characterizations of  minimal divisorial plane valuations of smooth projective surfaces with respect to ample divisors (Definition \ref{def:minimalvaluation}).

\begin{introtheorem}\label{Intro_thm:triangulo}
Let $D$ be an ample divisor on a smooth projective surface $S$ and $\nu_r$ a divisorial valuation of $S$. Set $\mathcal{C}_{\nu_r}=\{p_i\}_{i=1}^r$ the configuration of centers of $\nu_r,$ $\tilde{S}$ the surface defined by the divisorial valuation $\nu_r$ and $\nu$ the valuation defined by the flag $\{\tilde{S}\supset E_r\supset\{q\}\},$ where $E_r$ is the exceptional divisor which defines $\nu_r$ and $q$ is a closed point of $E_r$. Denote by $\mathfrak{T}_D(\nu)$ the triangle $\mathfrak{C}_{\nu}\cap \mathfrak{h}_D(\nu_r),$ $\mathfrak{C}_\nu$ being the cone generated by the value semigroup of $\nu$ and $\mathfrak{h}_D(\nu_r)$ the half-plane $\mathfrak{h}_D(\nu_r):=\{(x,y) \ | \ 0\leq x\leq\hat{\mu}_D(\nu_r)\}$, see Proposition \ref{prop:335}. Then, the following statements are equivalent:
\begin{enumerate}[(a)]
\item The divisorial valuation $\nu_r$ is minimal with respect to $D$.

\item The Newton-Okounkov body of $D$ with respect to $\nu$ coincides with the triangle $\mathfrak{T}_D(\nu)$.

\item The $\mathbb{R}$-divisor
$$
P_\mu:=D^* - \dfrac{\hat{\mu}_D(\nu_r)}{[\text{\emph{vol}}(\nu_r)]^{-1}} \sum_{i=1}^r \nu_r (\mathfrak{m}_i) \cdot E_i^*
$$
is nef and satisfies that $P_\mu^2=0$, where $\mathfrak{m}_i$ is the maximal ideal corresponding to the closed point $p_i$, $\nu_r(\mathfrak{m}_i)=\min\{\nu_r(f)\ | \ f\in\mathfrak{m}_i\setminus\{0\}\}$ and $E_i^*$ denotes the pull-back on $\tilde{S}$ of the exceptional divisor $E_i$ created by blowing up at $p_i$.

\item It holds that $\hat{\mu}_D(\nu_r)=\epsilon(D,\nu_r)[\text{\emph{vol}}(\nu_r)]^{-1}$.

\item The segment $\{[D^*-tE_r]\mid 0\leq t\leq \hat{\mu}_D(\nu_r)\}$ crosses only one Zariski chamber of the big cone of $\tilde{S}$.

\item It holds that $$\epsilon(D,\nu_r)=\sqrt{\frac{D^2}{[\text{\emph{vol}}(\nu_r)]^{-1}}}.$$

\end{enumerate}
\end{introtheorem}

Theorem \ref{Intro_thm:triangulo} is complemented with the following result (Theorem \ref{prop:submaximal_curve}) on submaximal curves in this context (see Definition \ref{def:submaximalcurves}).

\begin{introtheorem}\label{Intro_prop:submaximal_curve}
Let $D$ be an ample divisor on a smooth projective surface $S$ and $\nu_r$ a  divisorial valuation of $S$. If $\nu_r$ is not minimal with respect to $D$, then there exists an integral curve $C$ on $S$ with $\nu_r(\varphi_{C})>0,$ $\varphi_C$ being the germ of $C$ at the first point $p_1$ of the configuration of $\nu_r,$ such that
$$
\dfrac{D\cdot C}{\nu_r(\varphi_{C})}=\varepsilon(D,\nu_r)<\sqrt{\frac{D^2}{[\text{\emph{vol}}(\nu_r)]^{-1}}}.
$$
Moreover, if there exists an integral curve $C$ such that $\nu_r(\varphi_C)>0$ and $$ \dfrac{D\cdot C}{\nu_r(\varphi_C)}<\sqrt{\frac{D^2}{[\text{\emph{vol}}(\nu_r)]^{-1}}},$$ then the valuation $\nu_r$ is not minimal with respect to $D$.
\end{introtheorem}

Theorems \ref{Intro_thm:triangulo} and \ref{Intro_prop:submaximal_curve} could be regarded as analogue statements, in the valuative setting, to those related to the Nagata conjecture and involving Seshadri constants, submaximal curves and wonderful rays. Moreover, in Proposition \ref{prop:numsubmaximalcurves}, we prove that, when $\nu_r$ is not minimal with respect to $D,$ the Picard number of $S,\rho(S),$ is a bound on the number $n$ of submaximal curves which compute $\varepsilon(D,\nu_r)$ (Definition \ref{computes}). Furthermore, $2+\rho(S)-n$ is a bound on the number of Zariski chambers that the segment $\{[D^*-tE_r]\mid 0\leq t\leq \hat{\mu}_D(\nu_r)\}$ crosses (see Corollary \ref{cor:boundnumberZarCha_submaximalcurves}).

When $S= \mathbb{P}^2$ and $D=L$, we give in Corollary \ref{cor_val_Nag_conj} several equivalent statements to the valuative Nagata conjecture (for divisorial valuations):

\begin{introconjecture}\label{ConjectureP2}
Let $\nu_r$ be a very general divisorial valuation of $\; \mathbb{P}^2$ with normalized volume $\text{\emph{vol}}^N(\nu_r)$. Denote by $L$ a general projective line on $\mathbb{P}^2$ and by $\widetilde{\mathbb{P}^2}$ the surface defined by the valuation $\nu_r$. If $[\text{\emph{vol}}^N(\nu_r)]^{-1}\geq 9,$ then the following equivalent properties are satisfied:
\begin{enumerate}[(a)]
\item The divisorial valuation $\nu_r$ is minimal (with respect to $L$).

\item If $\nu$ is any exceptional curve valuation defined by a flag $\Big\{\widetilde{\mathbb{P}^2}\supset E_r\supset\{q\}\Big\}$, where $E_r$ is the exceptional divisor defining $\nu_r$ and $q$ is any closed point of $E_r$, then the Newton-Okounkov body of $L$ with respect to $\nu$ coincides with the triangle $\mathfrak{T}_L(\nu)$, see Proposition \ref{prop:335}.

\item The $\mathbb{R}$-divisor
$$
P_\mu:=L^* - \dfrac{\hat{\mu}(\nu_r)}{[\text{\emph{vol}}(\nu_r)]^{-1}} \sum_{i=1}^r \nu_r (\mathfrak{m}_i) \cdot E_i^*
$$
is nef and satisfies that $P_\mu^2=0$.

\item It holds that $\hat{\mu}(\nu_r)=\varepsilon(L,\nu_r)[\text{\emph{vol}}(\nu_r)]^{-1}$.

\item The segment $\{[L^*-tE_r]\mid 0\leq t\leq \hat{\mu}_L(\nu_r)\}$ crosses only one Zariski chamber of the big cone of $\widetilde{\mathbb{P}^2}$.

\item It holds that $$\epsilon(L,\nu_r)=\sqrt{\frac{1}{[\text{\emph{vol}}(\nu_r)]^{-1}}}.$$

\item There is no submaximal curve computing the constant $\epsilon(L,\nu_r)$.
\end{enumerate}
\end{introconjecture}

 In addition, in Corollary \ref{cor:epsilon_propertyP2} we prove that in this last case ($S=\mathbb{P}^2,D=L$) it holds that $\epsilon(L, \nu)= \frac{1}{\hat{\mu}(\nu)},$ and we determine a bound for the valuative Seshadri constant that depends only on the dual graph of $\nu_r$ and the value $\nu_r(\varphi_H)$ of a tangent line $H$ to $\nu_r$.

Apart from the introduction, the paper is structured in two sections. Section \ref{sec:pre} introduces the information to develop Section \ref{sec:valNagConj} which contains our results. It is divided into two subsections, Subsection \ref{subsec:seshadri_constant} on Seshadri constants and Subsection \ref{subsec:minimalvaluation} which studies minimality of valuations.

As mentioned, we conclude by posing the following valuative Nagata-type problem: To determine (only combinatorial if possible) conditions for a very general divisorial plane valuation of a smooth surface $S$ and an ample divisor $D$ on $S$ implying the minimality of the valuation with respect to $D$ (and, therefore, those equivalent statements given in Theorem A).

\section{Preliminaries}\label{sec:pre}
We devote this section to introduce some of the ingredients we will use in our results. We also recall some properties concerning them.

\subsection{Convex cones associated to surfaces}\label{subsec:cones}
Let $S$ be a smooth (complex) projective surface. Denote by Pic$(S)$ the Picard group of $S$ and by Num$(S)$ the quotient group given by the classes of Pic$(S)$ modulo numerical equivalence. Let Num$_\mathbb{R}(S)$ be the $\mathbb{R}$-vector space Num$(S)\otimes\mathbb{R}$ and $\cdot$ the  pairing induced by the intersection product in Pic$(S)$. Set $\text{NE}(S)\subset \text{Num}_\mathbb{R}(S)$ the effective cone (or cone of curves) of $S$, that is, the convex cone of $\text{Num}_\mathbb{R}(S)$ generated by the numerical equivalence classes of effective divisors on $S$, and denote its closure by $\overline{\text{NE}}(S)$, usually named the Mori cone of $S$. The divisors whose classes belong to $\overline{\text{NE}}(S)$ are called pseudoeffective. Moreover, denote by Nef$(S)$ (res\-pectively, Big$(S)$) the convex cone in $\text{Num}_\mathbb{R}(S)$ generated by the classes of nef (respectively,  big) divisors on $S$.

For a pseudoeffective divisor $D$ we write $D=P_D+N_D$ its Zariski decomposition, where $P_D$ (respectively, $N_D$) is the \emph{positive} (respectively, \emph{negative}) \emph{part of} $D$. The Zariski decomposition can also be defined for a pseudoeffective $\mathbb{Q}$ or $\mathbb{R}$-divisor (see \cite{Fuj} and \cite{Zar1}). \emph{Zariski chambers} are the subcones $\Sigma_P$ of the decomposition of the cone Big$(S)$:
$$
\mathrm{Big}(S)=\bigcup_{P ~\mathrm{big}~ \text{and}~ \mathrm{nef}} \Sigma_P,
$$
defined as
\begin{align*}
\Sigma_P &= \{[D] \in \mathrm{Big}(S) \ | \ \mathrm{Neg}(D)=\mathrm{Null}(P)\},\text{ where }\\
\end{align*}
\begin{align*}
\mathrm{Neg}(D)&:=\{C \ | \ C \mbox{ is an irreducible component of } N_D\}\text{ and } \\
\mathrm{Null}(D)&:=\{C\ | \ C \mbox{ is an  irreducible curve with } C \cdot D = 0\}.
\end{align*}
Notice that $\mathrm{Neg}(D)\cup\mathrm{Null}(D)\subset\mathrm{Null}(P_D)$, where the union may be disjoint or not.

\medskip

Fix an ample divisor $H$ on $S$ and define
$$
Q(S):=\Big\{[D]\in \text{Num}_\mathbb{R}(S) \ |\ [D]^2\geq 0 \mbox{ and } [D]\cdot [H] \geq 0\Big\}.
$$
Then, one has the chain of inclusions $\text{Nef}(S)\subset Q(S)\subset\overline{\text{NE}}(S)$ and also the following result:
\begin{proposition}\label{prop:Lemma1}
Let $S$ be a smooth projective surface and $D$ a divisor on $S$. Then,  the self-intersection of $D$ is positive if and only if $D\cdot x > 0$ for any $x\in Q(S)\setminus\{[0]\}.$
\end{proposition}
\begin{proof}
%Set $H$ an ample divisor on $S$.
By Hodge index theorem (\cite[Chapter V, Theorem 1.9]{Har}), there exists a basis
 $B=\{{\bf h}_0,{\bf h}_1,\ldots, {\bf h}_{n}\}$ of $\text{Num}_\mathbb{R}(S)$ such that ${\bf h}_0$ is a (real) multiple of $[H]$, ${\bf h}_0^2=1$, ${\bf h}_0\cdot {\bf h}_i=0$ and ${\bf h}_i\cdot {\bf h}_j=-\delta_{ij},$ for $1\leq i,j\leq n,$ where $\delta_{ij}$ is the Kronecker delta. Consider $x\in Q(S)\setminus\{0\}$ and let $(x_0,x_1,\ldots,x_n)$ be its vector of coordinates in the basis $B$. Since $x\in Q(S)\setminus\{0\}$, it holds that $x_0>0$. Set $\overline{y}=(y_1,y_2,\ldots,y_n)\in\mathbb{R}^n$ such that $y_i=x_i/x_0,$ for $1\leq i\leq n$. Notice that the condition $D\cdot x> 0$ for any $x\in Q(S)\setminus\{0\}$ is equivalent to the fact that $\sum_{i=1}^ny_id_i<d_0$ for any $\overline{y}\in\mathbb{R}^n$ satisfying the inequality $\sum_{i=1}^ny_i^2\leq 1,$ where $(d_0,d_1,\ldots,d_n)$ are the coordinates of $[D]$ in the basis $B$. Let $f:\mathbb{R}^n\to\mathbb{R}$ be the map  $f(y_1,y_2,\ldots,y_n) =\sum_{i=1}^nd_iy_i.$ Lagrange multipliers method proves that $\sqrt{\sum_{i=1}^nd_i^2}$ is the maximum of $f$ under the restriction $\sum_{i=1}^ny_i^2\leq 1.$ Consequently, $D^2>0$ (that is, $\sum_{i=1}^nd_i^2<d_0^2$) if, and only if, $\sum_{i=1}^ny_id_i<d_0$ for any $\overline{y}\in\mathbb{R}^n$ such that $\sum_{i=1}^ny_i^2\leq 1,$ which proves the result.
\end{proof}

Let $W$ be a convex cone of $\text{Num}_\mathbb{R}(S)$ and denote by $W^\vee$ its dual cone. For any element $z\in W$, set $\{z\}^\perp =\{x\in\text{Num}_\mathbb{R}(S)\ | \ z\cdot x=0\}.$
Subcones of the form $W\cap \{z\}^\perp,$ for some $z\in W^\vee$, are named \emph{faces} of $W$.  A face $\mathcal{F}$ of $W$ is said to be \emph{extremal} if, for any pair of elements $z_1,z_2\in W\setminus \{0\}$ such that $z_1+z_2\in\mathcal{F},$ it holds that $z_1,z_2\in \mathcal{F}.$ A one-dimensional extremal face is an \emph{extremal ray}.

By \cite[Lemma 1.22]{KolMor}, if $C$ is an irreducible curve on a surface $S$ such that $C^2\leq 0,$ then its numerical equivalence class belongs to the boundary of  $\overline{\text{NE}}(S)$. In addition, if $C^2< 0,$ its class generates an extremal ray of $\overline{\text{NE}}(S).$

The following result about extremal rays of $\overline{\text{NE}}(S)$ will be useful, see \cite[Lemma 1.2]{DeFer} or \cite[Proposition 3]{GalMon2} for a proof.
\begin{proposition}\label{prop:extremalray}
Every Cauchy sequence of extremal rays of the Mori cone $\overline{\text{\emph{NE}}}(S)$ converges to a ray in the boundary of $Q(S).$
\end{proposition}

\subsection{Plane valuations}

Let $R$ be a two-dimensional local regular ring and $\mathfrak{m}$ its maximal ideal. Denote by $K$ the quotient field of $R$ and set $K^*=K\setminus\{0\}$. A valuation $\nu$ of $K$ is a surjective map $\nu:K^*\to G,$ where $G$ is a totally ordered commutative group (the group of values of $\nu$), such that, for $f,g\in K^*$, it satisfies
$$
\nu (f+g)\geq \min\{\nu(f),\nu(g)\}\text{ and } \nu(fg)=\nu(f)+\nu(g).
$$
$R_\nu:=\{f\in K^* \ | \ \nu(f)\geq 0  \}\cup \{0\}$ is a local ring named the \emph{valuation ring} of $\nu,$ whose maximal ideal is $\mathfrak{m}_\nu:=\{f\in K^* \ | \ \nu(f)> 0  \}\cup \{0\}$. When the equality $R \cap \mathfrak{m}_\nu=\mathfrak{m}$ holds, $\nu$ is centered at $R$ and it is called a \emph{plane valuation}.

Zariski introduced an interesting geometric perspective of plane valuations.

\begin{theorem}\label{Thm_Zar_blowingups}
Plane valuations (up to equivalence) of $K$ centered at $R$ correspond one-to-one to simple sequences of point blowups of the scheme $\text{\emph{Spec} } R$.
\end{theorem}

Let $\nu$ be a plane valuation. By Theorem \ref{Thm_Zar_blowingups}, $\nu$ determines a simple sequence of blowups:
\begin{equation}\label{Eq_sequencepointblowingups_valuation}
\pi: \cdots \rightarrow T_n\xrightarrow{\pi_n} T_{n-1}\rightarrow \cdots \rightarrow T_1 \xrightarrow{\pi_1} T=\text{Spec }R,
\end{equation}
where $\pi_1$ is the blowup at the closed point $p=p_1\in T$ defined by the maximal ideal $\mathfrak{m}_1=\mathfrak{m}$ and $\pi_{i+1}, i\geq 1,$ is the blowup at the unique closed point $p_{i+1}\in T_i$ that belongs to the exceptional divisor $E_i$ created by $\pi_i$ and such that the plane valuation $\nu$ is centered at $\mathcal{O}_{T_i,p_{i+1}}$. This type of sequences, where each blowup is performed at a point of the last created exceptional divisor, are named \emph{simple}.

The set $\mathcal{C}_\nu:=\{p_i\}_{i\geq 1}$ is called the \emph{configuration of centers} of $\nu$. We say that $p_i$ is proximate to $p_j$, $i>j,$ (denoted $p_i\to p_j$) when $p_i \in E_j$, where $E_j$ is either the exceptional divisor obtained by blowing up $p_j$ or any of its strict transforms. Moreover, a center $p_i$ is called \emph{satellite} if there exists $j<i-1$ satisfying $p_i\to p_j$; otherwise, it is called \emph{free}.

The dual graph $\Gamma_\nu$ of a plane valuation $\nu$ is a (possibly infinite) labelled tree whose vertices represent the set of exceptional divisors created by the sequence of blowups \eqref{Eq_sequencepointblowingups_valuation} and two vertices are joined by an edge if and only if the associated divisors intersect. A label $i$ is attached to each vertex corresponding to the exceptional divisor $E_i$. Denote by $\{E_{l_j}\}_{0\leq j\leq g}$, with $l_0<l_1<\cdots<l_g$, the set of exceptional divisors (different from the last one when $\mathcal{C}_\nu$ is finite) whose associated vertices in $\Gamma_{\nu}$ have degree $1$ and, when $\mathcal{C}_\nu$ is finite, set $E_{l_{g+1}}$ the last exceptional divisor.

Spivakovsky in \cite{Spiv} classifies plane valuations into five types according to the shape of their dual graphs. We are interested in two of them: divisorial and  exceptional curve valuations (using the terminology of \cite{FavJon1}).

Specifically, a \emph{divisorial} valuation is a plane valuation whose configuration of centers is finite. Usually we denote these valuations by $\nu_n,$ where $n$ is the cardinality of its configuration of centers. The group of values of any divisorial valuation is isomorphic to $\mathbb{Z}$. In addition,
$\nu_n$ can be computed as $\nu_n(f)=c \cdot \text{ord}_{E_n}(f),$ $f\in K^*$, where $c\in\mathbb{R}\setminus\{0\}$ and $E_n$ is the last exceptional divisor created by the sequence of blowups \eqref{Eq_sequencepointblowingups_valuation} associated to $\nu_n$. Different values of $c$ give rise to equivalent valuations. Throughout this work, we assume $c=1$ when we refer to $\nu_n$. Set $\nu_n(\mathfrak{m}):=\min\{\nu_n(f)\ | \ f\in\mathfrak{m}\setminus\{0\}\};$ the equivalent valuation to $\nu_n$ given by $c=1/\nu_n(\mathfrak{m})$  is called the \emph{normalized valuation} of $\nu_n$ and denoted by $\nu_n^N$.

An \emph{exceptional curve valuation} is a plane valuation whose configuration of centers is infinite and there exists a center  $p_r$ such that all the centers $p_i$, with $i>r$, are proximate to $p_r$. Its  group of values is isomorphic to $\mathbb{Z}^2_{\text{lex}}$ (lexicographically ordered).

\subsubsection{Some invariants of plane divisorial valuations}\label{subsec:invariantsvaluation}

In this subsection we recall two of the invariants which have been considered for studying plane divisorial valuations \cite{Spiv,DelGalNun}.

Let $\nu_n$ be a divisorial valuation and $\mathcal{C}_{\nu_n}=\{p_i\}_{i=1}^n$ its configuration of centers. Denote by $\mathfrak{m}_i$ the maximal ideal corresponding to the closed point $p_i$, $i\geq 1.$ Define $\nu_n(\mathfrak{m}_i):=\min\{\nu_n(f)\ | \ f\in\mathfrak{m}_i\setminus\{0\}\}$. The  \emph{sequence of values} of $\nu_n$ is the ordered set $\left(\nu_n(\mathfrak{m}_i)\right)_{i=1}^n$. It satisfies the proximity equalities \cite[Theorem 8.1.7]{Cas}:
\begin{equation}\label{proximity_equalities}
\nu_n(\mathfrak{m}_i)=\sum_{p_j\to p_i} \nu_n(\mathfrak{m}_j), \, i\geq 1,
\end{equation}
when the set $\{p_j\in\mathcal{C}_{\nu_n} \ | \ p_j\to p_i\}$ is not empty. In addition, using the \emph{Noether formula} for valuations \cite[Theorem 8.1.6]{Cas}, the value $\nu_n(f),f\in R,$ can be computed from the sequence of values of $\nu_n$, the multiplicity of $f$ at $p_1,$ ${\rm mult}_{p_1}(f),$ and the multiplicities ${\rm mult}_{p_i}(f)$ of the strict transforms of $f$ on $T_{i-1}$ at $p_i$, $i\geq 2$:
\begin{equation}\label{noether_formula}\nu_n(f)=\sum_{i=1}^n \nu_n(\mathfrak{m}_i){\rm mult}_{p_i}(f).
\end{equation}

Now, let $\nu$ be an exceptional curve valuation. Then $\mathcal{C}_{\nu}=\{p_i\}_{i\geq 1}$  and there is a positive integer $r$ such that $p_i\to p_r$ for all $i > r$. It holds that  $\nu(\mathfrak{m}_r)=(a,b)$ and $\nu (\mathfrak{m}_i)=(0,c)$ for all $i>r$, for some $a,b,c\in\mathbb{Z}$, $a,c>0$; the remaining elements of its sequence of values (defined as in the divisorial case but with infinitely many elements) are uniquely determined by the above pairs by using the proximity equalities, which also work for exceptional curve valuations (see \cite{DelGalNun}).

Consider again a divisorial valuation $\nu_n$. Denote by $\varphi_i, 1 \leq i \leq n,$ an analytically irreducible germ of curve at $p$ whose strict transform on $T_i$ is transversal to $E_i$ at a general point of the exceptional locus. Define $\overline{\beta}_i(\nu_n):=\nu_n(\varphi_{l_i})$, $0\leq i\leq g+1$. Then, the ordered set $(\overline{\beta}_i(\nu_n))_{i=0}^{g+1}$ is named the \emph{sequence of maximal contact values} of $\nu_n$. By \cite[Remark 6.1]{Spiv}, the set $(\overline{\beta}_i(\nu_n))_{i=0}^{g}$ minimally generates the semigroup of values of $\nu_n,$ $S(\nu_n)=\nu_n(R\setminus\{0\})$. By \cite[Remark 2.3]{GalMonMoyNic2}, the value $\overline{\beta}_{g+1}(\nu_n)$ coincides with the inverse of the volume of $ \nu_n,$ $[\text{vol}(\nu_n)]^{-1},$ which is defined as
\begin{equation}\label{Eq_igualdadvolumenvaloracion_cap1}
\text{vol}(\nu_n):=\lim_{\alpha\to \infty}\dfrac{\text{length}(R/\mathcal{P}_\alpha)}{\alpha^2 /2},
\end{equation}
$\mathcal{P}_\alpha$ being the ideal $\mathcal{P}_\alpha=\{f\in R\setminus\{0\} \ | \ \nu_n(f)\geq \alpha\}\cup\{0\}$. In addition, since ${\rm mult}_{p_i}(\varphi_n)$ $=\nu_n(\mathfrak{m}_i)$ for all $i=1,2,\ldots,n$, it holds that \begin{equation}\label{betabarra_gmasuno}\overline{\beta}_{g+1}(\nu_n)=\sum_{i=1}^n\nu_n(\mathfrak{m}_i)^2,\end{equation}
as a consequence of the Noether formula (\ref{noether_formula}).
Finally, we define the \emph{normalized volume} of $\nu_n,\text{vol}^N(\nu_n),$ as the volume of the normalized valuation $\nu_n^N,$ that is, $$ \text{vol}^N(\nu_n):=\text{vol}(\nu_n^N)=\frac{\overline{\beta}_0^2(\nu_n)}{\overline{\beta}_{g+1}(\nu_n)}. $$

\subsection{Seshadri-type constants for divisorial valuations.}\label{subsec:mupico}

Let $S$ be a smooth projective surface and $p$ a closed point of $S$. Let $R=\mathcal{O}_{S,p}$ and consider $\nu_n$ a divisorial valuation of the quotient field of $R$ centered at $R.$ As showed in \eqref{Eq_sequencepointblowingups_valuation}, $\nu_n$ defines   a  finite simple sequence of $n$ blowups, $n\geq 1$,
\begin{equation}\label{Eq_sequencepointblowingups_divvaluation}
\pi: \tilde{S}:= S_n\xrightarrow{\pi_n} S_{n-1}\rightarrow \cdots \rightarrow S_1 \xrightarrow{\pi_1} S,
\end{equation}
where $p_1=p$ is the center of the first blowup and $S_i,1\leq i\leq n,$ is a smooth projective surface. The valuation $\nu_n$ is also defined by the sequence \eqref{Eq_sequencepointblowingups_divvaluation}. For simplicity, we say that $\nu_n$ is a \emph{divisorial valuation of} $S$.

From now on we will keep the notation of the preceding section, denoting by $E_i$ the exceptional divisor created by the blowup $\pi_i$, $1\leq i\leq n$. In addition, abusing the notation, we will also denote by $E_i$ its strict transform on any of the surfaces $S_j$ with $j>i$ (in particular, on $\tilde{S}$).

For a big divisor $D$ on $S,$ the \emph{Seshadri-type constant} $\hat{\mu}_D(\nu_n)$ \emph{of $\nu_n$ relative to $D$} is defined as follows:
\begin{equation}\label{eq:defmupico}
\hat\mu_D(\nu_n):=\lim_{m\to \infty}\dfrac{\max\{\nu_n(f) \ | \ f\in H^0(S,\mathcal{O}_{S}(mD))\}}{m}.
\end{equation}
This constant was introduced in \cite{BouKurMacSze} in greater generality. Notice that
$$\hat\mu_D(\nu_n)=\sup_{m\geq m_0}\dfrac{\max\{\nu_n(f) \ | \ f\in H^0(S,\mathcal{O}_{S}(mD))\}}{m}$$
for any choice of $m_0\geq 1$ such that $H^0(X,{\mathcal O}_{S}(mD))\neq 0$ for all $m\geq m_0$ (see \cite[Lemma 2.6]{BouKurMacSze}).
Set $\hat{\mu}^N_D(\nu_n):=\hat{\mu}_D(\nu_n^N),$ where $\nu_n^N$ is the normalization of $\nu_n.$ The Seshadri-type constant (or its normalized version) admits the following bound (see  \cite{BouKurMacSze}):

\begin{equation}\label{Eq_desigualdadmupico}
\hat\mu_D(\nu_n)\geq \sqrt{\dfrac{\text{vol}_{S}(D)}{\text{vol}(\nu_n)}} \text{ or, equivalently, }\hat\mu_D^N(\nu_n)\geq \sqrt{\dfrac{\text{vol}_{S}(D)}{\text{vol}^N(\nu_n)}},
\end{equation}
where
\begin{equation}\label{Eq_igualdadvolumendivisor}
\text{vol}_{S}(D):=\limsup_{m\to\infty}\dfrac{h^0(S,mD)}{m^2/2}
\end{equation}
is named \emph{the volume of }$D$.

\begin{definition}\label{def:minimalvaluation}
A divisorial valuation $\nu_n$ of $S$ is called \emph{minimal with respect to a big divisor} $D$ on $S$ if
\begin{equation*}
\hat\mu_D(\nu_n)= \sqrt{\dfrac{\text{vol}_{S}(D)}{\text{vol}(\nu_n)}} \text{ or, equivalently, if }\hat\mu_D^N(\nu_n)= \sqrt{\dfrac{\text{vol}_{S}(D)}{\text{vol}^N(\nu_n)}}.
\end{equation*}
\end{definition}

\begin{remark}
Equality \eqref{eq:defmupico} shows that the Seshadri-type constant $\hat{\mu}_D(\nu_n)$ satisfies the homogeneous property.  That is, $$\hat{\mu}_{dD}(\nu_n)=d\hat{\mu}_{D}(\nu_n),$$ for any positive integer $d$.  As a consequence, a divisorial valuation $\nu_n$ of $\mathbb{P}^2$ is minimal with respect to any big divisor $D$ on $\mathbb{P}^2$ if and only if it is minimal in the sense of \cite{DumHarKurRoeSze} and \cite{GalMonMoy}.
\end{remark}

\subsection{Newton-Okounkov bodies of exceptional curve valuations}\label{subsec:Nobodies}

Let $Y$ be a smooth (complex) projective surface. A \emph{flag} of $Y$ is a sequence
$$
C_\bullet:=\{Y\supset C\supset \{q\}\},
$$
where $C$ is a smooth irreducible curve on $Y$ and $q$ a closed point of $C$. The point $q$ is called  \emph{center} of $C_\bullet.$

In what follows we will consider flags of the following type:
\begin{equation}\label{eq:flag}
E_\bullet:=\lbrace \tilde{S}=S_r\supset E_r \supset \{p_{r+1}\}   \rbrace,
\end{equation}
where $\tilde{S}=S_r$ is the surface defined by a finite simple sequence of blowups as in \eqref{Eq_sequencepointblowingups_divvaluation}, $E_r$ is the last created exceptional divisor and $p_{r+1}$ is the center of $E_\bullet$. A flag as above defines and is defined by an exceptional curve valuation $\nu$ (see \cite{GalMonMoyNic2}). Its configuration of centers $\mathcal{C}_\nu=\{p_i\}_{i\geq 1}$ satisfies that the set $\{p_i\}_{i=1}^r$ is the configuration of centers of the divisorial valuation $\nu_r$ defined by the exceptional divisor $E_r$, and the remaining centers $p_i$, $i>r,$ are proximate to $p_r$. Notice that, if $p_{r+1}$ is a satellite point, there exists a positive integer $\eta\neq r$ such that $p_{r+1}\in E_\eta\cap E_r.$

Following \cite[Section 3.2]{GalMonMoyNic2}, the exceptional curve valuation associated to $E_\bullet,$ $\nu:=\nu_{E_\bullet},$ is a valuation of rank $2$ such that, for $f\in\mathcal{O}_{S,p_1}, \nu_{E_\bullet}(f)=(v_1(f),v_2(f))$ with $v_1(f):=\nu_r(f)$ and $v_2(f):=\nu_\eta(f)+ \sum_{p_i\to p_r}\text{mult}_{p_i}(f),$ where $\nu_\eta$ is the divisorial valuation defined by the exceptional divisor $E_\eta$ when $p_{r+1}$ is satellite and $\nu_\eta (f) = 0$, otherwise. The value group of $\nu$ is $\mathbb{Z}^2,$ and $\nu(\mathfrak{m}_r)=(1,0)$ and $\nu(\mathfrak{m}_{r+1})=(0,1).$

For $j=r,\eta,$ consider the sequence $(\overline{\beta}_i(\nu_j))_{i=0}^g$ of maximal contact values of $\nu_j$. Denote by $S_\nu$ the semigroup of values of the exceptional curve valuation $\nu$. Write $g^*+1$ the minimal number of generators of $S_\nu$. If $p_r$ and $p_{r+1}$ are satellite points, then it holds that $g^*=g$. Otherwise, $g^*=g+1$. The semigroup $S_\nu$ is minimally generated by the set $\{\overline{\beta}_i(\nu)\}_{i=0}^{g^*},$ where $\overline{\beta}_i(\nu)=(\overline{\beta}_i(\nu_r),\overline{\beta}_i(\nu_\eta)),$ $0\leq i\leq g^*,$  (respectively, $\overline{\beta}_i(\nu)=(\overline{\beta}_i(\nu_r),0),0\leq i\leq g^*-1,$ and $\overline{\beta}_{g^*}(\nu)=(\overline{\beta}_{g+1}(\nu_r),1)$), when $p_{r+1}$ is a satellite (respectively, free) point (see \cite{DelGalNun,GalMonMoyNic2}).

The following result uses the following ordering on the vertices of the dual graph of $\nu_r:$  $\alpha\preccurlyeq\beta$ if there exists a path on the graph $\Gamma_{\nu_r}$ from $1$ to $\beta$ passing through $\alpha$.

\begin{lemma}\cite[\emph{Proposition 2.5 and Lemma 3.9}]{GalMonMoyNic2}\label{lem:336}
Keep the notation introduced before. Assume that $p_{r+1}$ is the satellite point $E_\eta\cap E_r, \eta\neq r$. Then,
\begin{itemize}
\item[(a)] It holds that
\begin{equation*}
\overline{\beta}_{g+1}(\nu_r)=\left|\dfrac{\overline{\beta}_{g^*}(\nu_\eta)}{\overline{\beta}_{g^*}(\nu_r)}-\dfrac{\overline{\beta}_{0}(\nu_\eta)}{\overline{\beta}_{0}(\nu_r)}\right|^{-1}.
\end{equation*}
\item[(b)] Denote by $\varphi_\eta$ an analytically irreducible element in $\mathcal{O}_{S,p},$ whose strict transform on $\tilde{S}$ transversally meets $E_\eta$ at non-singular point.
\begin{itemize}
\item[(1)] If $\eta \preccurlyeq r,$  it holds
$$
\nu_r(\varphi_\eta)=\overline{\beta}_{g+1}(\nu_r)\cdot \dfrac{\overline{\beta}_{g^*}(\nu_\eta)}{\overline{\beta}_{g^*}(\nu_r)} \text{ and } \nu_r(\varphi_\eta) + 1 =\overline{\beta}_{g+1}(\nu_r)\cdot \dfrac{\overline{\beta}_{0}(\nu_\eta)}{\overline{\beta}_{0}(\nu_r)}.
$$
\item[(2)] If $\eta \not\preccurlyeq r,$  it holds
$$
\nu_r(\varphi_\eta)=\overline{\beta}_{g+1}(\nu_r)\cdot \dfrac{\overline{\beta}_{0}(\nu_\eta)}{\overline{\beta}_{0}(\nu_r)} \text{ and } \nu_r(\varphi_\eta) + 1 =\overline{\beta}_{g+1}(\nu_r)\cdot \dfrac{\overline{\beta}_{g^*}(\nu_\eta)}{\overline{\beta}_{g^*}(\nu_r)}.
$$
\end{itemize}
\end{itemize}
\end{lemma}

\begin{definition}
Let $D$ be a big divisor on a smooth projective surface $S$ and $\nu$ an exceptional curve valuation of $S$. We say that $\nu$ is \emph{minimal with respect to} $D$ when its first component $\nu_r$ is minimal with respect to $D.$
\end{definition}

Now we introduce the notion of Newton-Okounkov body in our context and some properties that will be of interest in this paper.

\begin{definition}\label{def:NObody}
Let $\nu$ be an exceptional curve valuation of $S$ and $D$ a big divisor on $S$. The \emph{Newton-Okounkov body of $D$ with respect to $\nu$} is defined as
$$
\Delta_{\nu}(D):=\overline{\bigcup_{m\geq 1}\bigg\{\dfrac{\nu (f)}{m}\ \Big| \ f\in H^0(S,mD)\setminus \{0\}\bigg\}},
$$
where the upper line means closed convex hull in $\mathbb{R}^2$.
\end{definition}
Sometimes $\Delta_{\nu}(D)$ is denoted as $\Delta_{\nu_{E_\bullet}}(D^*)$, where $D^*$ is the pull-back of $D$ on the surface $\tilde{S}$. In addition, $\Delta_{\nu}(D)$ is  a  polygon (see \cite{KurLozMac}) and
$$
\text{vol}(D)=\text{vol}_{\tilde{S}}(D^*)=2\,\text{vol}_{\mathbb{R}^2}(\Delta_{\nu}(D)),
$$
where $\text{vol}_{\mathbb{R}^2}$ means Euclidean area (see \cite{LazMus}).

\medskip

The following proposition is a reformulation of Theorem 6.4 in \cite{LazMus} within the framework of this paper.

\begin{proposition}\label{pro:LM-alphabeta}
Let $D$ be a big and nef divisor on a smooth projective surface $S$. Consider an exceptional curve valuation $\nu$ of $S$ and set $E_\bullet:=\lbrace \tilde{S}=S_r\supset E_r \supset \{p_{r+1}\}   \rbrace$ the flag that $\nu$ defines. Let $D_t=P_t+N_t$ be the Zariski decomposition of the $\mathbb{R}$-divisor $D_t:=D^*-t E_r$, $t\in\mathbb{R}$, and let $\mu =\mu(D^*,E_r):= \sup\{t>0 \ | \ D^* - t E_r \mbox{ is big}\}$. Then, the Newton--Okounkov body of $D$ with respect to $\nu$ is given by
\[
\Delta_{\nu}(D) = \big\{(t,y) \in \mathbb{R}^2 \ | \ 0 \leq t \leq \mu, \, \alpha(t) \leq y \leq \beta(t) \big\} \,,
\]
where
\[
\alpha(t)=\ord_{p_{r+1}}(N_t|_{E_r}), \quad \beta(t) = \alpha(t)+(E_r \cdot P_t).
\]
\end{proposition}

\medskip

As a consequence we have a geometrical interpretation of the Seshadri-type constant $\hat{\mu}_D(\nu_r)$. Although this interpretation is well-known, we include it for the reader's convenience.

\begin{proposition}\label{prop:mupico_geometry}
Let $D$ be a big divisor on a smooth projective surface $S$ and $\nu_r$ a divisorial valuation of $S$ defining a surface $\tilde{S}.$ Denote by $E_r$ the exceptional divisor which defines $\nu_r$ and by $D^*$ the total transform of the divisor $D$ on $\tilde{S}.$ Then,
 \begin{equation}\label{Eq_mupico_geo_version}
\hat{\mu}_D(\nu_r)=\muu:=\sup\Big\{t> 0 \ | \ D^*-tE_r \text{ is big}\Big\}.
\end{equation}
\end{proposition}
\begin{proof}
 Equality \eqref{eq:defmupico}, Definition \ref{def:NObody} and Proposition \ref{pro:LM-alphabeta} show the inequality $\hat{\mu}_D(\nu_r)\leq\muu$. For the opposite inequality, reasoning by contradiction, we assume that $\hat{\mu}_D(\nu_r)$ $<\muu.$ Then, there exist positive integers $a,b$ and $m$ such that
$$
\hat{\mu}_D(\nu_r)<\frac{a}{b}\leq \muu \text{ and } H^0(\tilde{S},{\mathcal O}_{\tilde{S}}(mbD^*-maE_r))\neq 0.
$$
This implies the existence of a curve $C\in |mbD|$ such that $\nu_r(\varphi_C)\geq ma$ and thus,
$$
\hat{\mu}_D(\nu_r)<\frac{a}{b}\leq\frac{\nu_r(\varphi_C)}{b},
$$
which is a contradiction.
\end{proof}

The following result, taken from \cite[Lemma 3.7]{GalMonMoyNic2}, shows that the Newton--Okounkov body of a big divisor $D$ with respect to an exceptional curve valuation $\nu$ is always included in a distinguished triangle.

\begin{proposition}\label{prop:335}
Keep the above notation  and let $\nu$ be an exceptional curve valuation defined by a flag $E_\bullet$ as in \eqref{eq:flag}. The Newton--Okounkov body $\Delta_\nu(D)$ of a big divisor $D$ on $S$ with respect to $\nu$ is contained in the triangle $\mathfrak{T}_D(\nu)$ whose vertices are
$$
(0,0),\ \ \left(\hat{\mu}_D(\nu_r),\dfrac{\hat{\mu}_D(\nu_r)\overline{\beta}_0(\nu_\eta)}{\overline{\beta}_0(\nu_r)}\right) \ \ \text{ and } \ \ \left(\hat{\mu}_D(\nu_r),\dfrac{\hat{\mu}_D(\nu_r) \overline{\beta}_{g^*}(\nu_\eta)}{\overline{\beta}_{g^*}(\nu_r)}\right)
$$
if $p_{r+1}$ is the satellite point $E_\eta \cap E_r$, $\eta \neq r$, and
$$
(0,0),\ \ \left(\hat{\mu}_D(\nu_r) ,0\right) \ \ \text{ and } \ \ \left(\hat{\mu}_D(\nu_r),\dfrac{\hat{\mu}_D(\nu_r)}{\overline{\beta}_{g+1}(\nu_r)}\right),
$$
otherwise ($p_{r+1}$ is free).
\end{proposition}

The following result describes an interesting property concerning the shape of Newton-Okounkov bodies with respect to exceptional curve valuations. It can be deduced from results by  K\"uronya, Lozovanu and Maclean \cite{KurLozMac}, and Roé and Szemberg \cite{RoeSze}. See \cite[Theorem 2.6]{MoyFerNicRoe} for this formulation.

\begin{proposition}\label{pro:KLM}
Let $D$ be a big divisor on a smooth projective surface $S$ and consider the flag $E_\bullet=\{\tilde{S}\supset E_r \supset \{p_{r+1}\} \}$. Consider the $\mathbb{R}$-divisors $D_t=D-tE_r, 0\leq t\leq \hat{\mu}_{D}(\nu_r),$ with Zariski decompositions $D_t=P_t+N_t$, and let $\alpha(t)$, $\beta(t)$ be the maps given in Proposition \ref{pro:LM-alphabeta}. Then, either the map $\alpha(t)$ or the map $\beta(t)$ has a point of non-differentiability at $t=t_0$ if and only if the negative part of $D_t$ acquires one or more irreducible components at $t_0$, i.e., for every $\varepsilon>0$ the supports of $N_{t_0}$ and $N_{t_0+\varepsilon}$ differ.
\end{proposition}

\begin{remark}\label{remark:thmKLM}
As explained in \cite[page 7]{MoyFerNicRoe} (see also \cite[Remark 1.5]{KurLoz1}), we emphasize the following fact: the negative part of $D_t$ acquires one or more irreducible components at $t_0$ if and only if $[D_{t}],$ where $ t_0-\epsilon < t < t_0+\epsilon$ and $\epsilon >0,$ crosses from one Zariski chamber to another.
\end{remark}

\section{The valuative Nagata conjecture for plane divisorial valuations}\label{sec:valNagConj}

The valuative Nagata conjecture \cite{DumHarKurRoeSze,GalMonMoy} states that, if $\nu_r$ is a very general real plane valuation of $\mathbb{P}^2$ such that $[\text{vol}^N(\nu_r)]^{-1}\geq 9,$ then $\nu_r$ is minimal. In \cite{GalMonMoy} it was proved that this conjecture implies the classical Nagata conjecture and it is implied by the Greuel-Lossen-Shustin conjecture \cite[Conjecture 4.7.11]{GreuelBook}. This section considers a smooth (complex) projective surface $S$, an ample divisor $D$ on $S$ and a divisorial valuation $\nu_r$ of $S$. As a main result, we prove the existence of several equivalent statements to the minimality of $\nu_r$ with respect to $D,$ which provides several equivalent statements to the valuative Nagata conjecture in terms of interesting algebraic and geometric tools.

The Seshadri constant is an important invariant  very related to the classical Nagata conjecture. In the next subsection we introduce a notion of Seshadri constant for a divisorial valuation of a smooth projective surface. We will see that this constant naturally extends the concept of one point Seshadri constant.

\subsection{Seshadri constant for a divisorial valuation of a smooth projective surface}\label{subsec:seshadri_constant}
Let $S$ be a smooth projective surface over the field of complex numbers and keep the notation introduced in the previous section.

\begin{definition}\label{def:sdnu}
Let $D$ be a nef divisor on $S$ and consider a divisorial valuation $\nu_r$ defined by a sequence of $r$ blowups, $\pi,$ as in \eqref{Eq_sequencepointblowingups_divvaluation}. Then, the \emph{Seshadri constant of $D$ with respect to the valuation $\nu_r$} is defined as the value
$$
\varepsilon(D,\nu_r)=\varepsilon(S,D,\nu_r):=\sup\Bigg\{t\geq 0 \ \Big| \ D^*-t\sum_{i=1}^r\nu_r(\mathfrak{m}_i)E_i^*\text{ is nef on }\tilde{S}\Bigg\},
$$
where $\tilde{S}$ is the surface defined by $\nu_r$ and $D^*$ (respectively, $E_i^*$) the total transform on $\tilde{S}$ of $D$ (respectively, the exceptional divisor $E_i$ defined by $\pi$).
\end{definition}

\begin{remark}
Notice that, if $r>1$, the $\mathbb{R}$-divisor $D^*-t\sum_{i=1}^r\nu_r(\mathfrak{m}_i)E_i^*$ is nef but not ample for the values $t<\varepsilon(D,\nu_r)$ since its intersection product with the strict transform of any exceptional divisor $E_i,1\leq i<r,$ vanishes. As a result, under these conditions, the class in $\text{Num}_\mathbb{R}(\tilde{S})$ of $D^*-t\sum_{i=1}^r\nu_r(\mathfrak{m}_i)E_i^*$ belongs to the boundary of Nef$(\tilde{S})$.
\end{remark}

\begin{remark}
The Seshadri constant $\varepsilon(D,\nu_r)$ satisfies the homogeneous property:
$$
\varepsilon(m\,D,\nu_r)=m\,\varepsilon(D,\nu_r), \text{ for } m\in\mathbb{Z} \text{ and } m\geq 0.
$$
\end{remark}

\begin{remark}\label{rem:Bounds_Sdnu}
By \eqref{betabarra_gmasuno}, the following upper bound holds:
\begin{equation}\label{eq:upperbound_seshadri}
\varepsilon(D,\nu_r)\leq \sqrt{\dfrac{D^2}{\overline{\beta}_{g+1}(\nu_r)}}.
\end{equation}
In addition, when $r=1$,  $\varepsilon(D,\nu_r)=\varepsilon(D,p),$  $\varepsilon(D,p)$ being the Seshadri constant of $D$ at the point $p\in S$ where the divisorial valuation $\nu_r$ is centered, cf. \cite[Definition 5.1.1]{Laz1}.
\end{remark}

\begin{definition}\label{def:seshadri_maximal}
The Seshadri constant $\varepsilon(D,\nu_r)$ is said to be \emph{maximal} whenever Inequality \eqref{eq:upperbound_seshadri} is an equality. Otherwise, $\varepsilon(D,\nu_r)$ is named \emph{submaximal}.
\end{definition}

The following result provides an equivalent definition for $\varepsilon(D,\nu_r).$ The idea of the proof comes from \cite[Exercise I.4.9]{Harb2}.

\begin{lemma}\label{lemm:infimo_seshadri}
Let $D$ be a nef divisor on $S$ and $\nu_r$ a divisorial valuation of $S$. Then
$$
\begin{array}{rcl}
\varepsilon(D,\nu_r)&=&s_1:=\inf\Bigg\{\dfrac{D\cdot C}{\nu_r(\varphi_C)}\ \Big| \ C\text{ is a curve on }S \text{ such that }\nu_r(\varphi_{C})>0\Bigg\}\\[7mm]
&=&s_2:=\inf\Bigg\{\dfrac{D\cdot C}{\nu_r(\varphi_C)}\ \Big| \ C\text{ is an integral curve on }S \text{ such that }\nu_r(\varphi_{C})>0\Bigg\}.%\\[2mm]
\end{array}
$$
\end{lemma}
\begin{proof}
We first prove the equality $\varepsilon(D,\nu_r)=s_1.$ Consider the $\mathbb{R}$-divisor $$D_{s_1}=D^*-s_1\sum_{i=1}^r\nu_r(\mathfrak{m}_i)E_i^*$$ on $\tilde{S}$. It holds that $D_{s_1}$ is nef. Indeed, one has that
$$
s_1\leq \dfrac{D^*\cdot C^*}{\nu_r(\varphi_C)}
$$
for any curve $C$ on $S$ such that $\nu_r(\varphi_C)>0,$ where $C^*$ denotes the total transform of $C$ on $\tilde{S}$. Then, by \eqref{noether_formula}, $D_{s_1}\cdot \tilde{C}\geq 0$ for every curve $C$ on $S,$ since $D$ is nef  on $S$. Moreover, by \eqref{proximity_equalities}, $D_{s_1}\cdot E_i\geq 0$ for $1\leq i\leq r$. Therefore, we deduce that $D_{s_1}$ is nef and $s_1\leq \varepsilon(D,\nu_r).$ To finish the proof of our first equality, it remains to show that $s_1\geq \varepsilon(D,\nu_r).$ Reasoning by contradiction, consider a positive real number $t$ such that $s_1<t<\varepsilon(D,\nu_r).$ Then,
$$
\left(D^* - t\sum_{i=1}^r\nu_r(\mathfrak{m}_i)E_i^*\right)\cdot \tilde{C}\geq 0 \text{ and equivalently } D^*\cdot C^*\geq t\,\nu_r(\varphi_C),
$$
for all those curves $C$ on $S$ such that $\nu_r(\varphi_C)$ is positive. Consequently, $s_1\geq t$ and we obtain the desired contradiction.

To conclude the proof, we are going to show that $s_1=s_2$. Clearly it suffices to prove that $s_1\geq s_2.$ Set $C$ a curve on $S$ such that $\nu_r(\varphi_C)$ is positive. This curve can be written as $C=\sum_{i}\alpha_iC_i$, for finitely many integral curves $C_i$ and positive integers $\alpha_i$. As a consequence, we have
$$
\dfrac{D\cdot C}{\nu_r(\varphi_C)}=\dfrac{\sum_i \alpha_i\,D\cdot C_i}{\sum_i \alpha_i\,\nu_r(\varphi_{C_i})}\geq \min\Bigg\{\dfrac{D\cdot C_i}{\nu_r(\varphi_{C_i})}\Bigg\}\geq s_2,
$$
where the first inequality follows from the fact that
$$
\dfrac{a_1+a_2}{b_1+b_2}\geq \min\Bigg\{\dfrac{a_1}{b_1},\dfrac{a_2}{b_2}\Bigg\}, \text{ for any positive real numbers } a_1,a_2,b_1\text{ and } b_2.
$$
\end{proof}

\begin{definition}\label{computes}
Keep the notation of the above lemma. We say that an integral curve $C$ on $S$ \emph{computes the Seshadri constant} $\varepsilon(D,\nu_r)$ if $\nu_r(\varphi_C)>0$ and
$\varepsilon(D,\nu_r)=\dfrac{D\cdot C}{\nu_r(\varphi_C)}.$
\end{definition}

The following result is an extension of the Seshadri criterion for ampleness on surfaces (see \cite[Theorem 1.4.13]{Laz1}) by considering divisorial valuations.

\begin{theorem}\label{thm_criterion_amplitude}
Let $S$ be a smooth projective surface. A divisor $D$ on $S$ is ample if and only if the valuative Seshadri constant $\varepsilon(D,\nu_r)$ is positive for any point $p\in S$ and any divisorial valuation $\nu_r$ of $S$ centered at $\mathcal{O}_{S,p}.$
\end{theorem}
\begin{proof}
Fix a divisorial valuation $\nu_r$ of $S$ as in the statement. Denote by $\tilde{S}$ the surface defined by $\nu_r$. Fix also an ample divisor $H$ on $\tilde{S}$ and consider the convex cone $$Q(\tilde{S}):=\{x\in \text{Num}_\mathbb{R}(\tilde{S}) \ |\ x^2\geq 0 \mbox{ and } x\cdot [H] \geq 0\}$$ defined in Subsection \ref{subsec:cones}. Let $D$ be an ample divisor on $S,$ we are going to prove that $\varepsilon(D,\nu_r)>0.$ Consider the set of divisors on $\tilde{S}$ $\; \{F_m:=mD^*-\sum_{i=1}^r\nu_r(\mathfrak{m}_i)E_i^*\}$,  where $m$ runs over the positive integers and $D^*$ is the total transform of $D$ on $\tilde{S}$. Then, for a sufficiently large positive integer $m,$ the divisor $F_m$ on $\tilde{S}$ is nef. Indeed, the intersection product of the divisor $F_m$ and the strict transform of any exceptional divisor is non-negative by \eqref{proximity_equalities}. In addition, since $D$ is ample, we can assume that the divisor $F_m$ has positive self-intersection considering a suitable positive integer $m$. Therefore, by Propositions \ref{prop:extremalray} and \ref{prop:Lemma1}, the set of integral curves on $\tilde{S}$ whose intersection product with $F_m$ is negative is finite. Consequently, increasing the value $m$ if it is necessary, $F_m$ will be nef (because $D$ is ample and these curves are the strict transforms of curves on $S$).

Now, consider any integral curve $C$ on $S$ such that $\nu_r(\varphi_C)>0$. Its strict transform $\tilde{C}$ on $\tilde{S}$ is linearly equivalent to  $C^*-\sum_{i=1}^r\text{mult}_{p_i}(\varphi_C)E_i^*$. Since $F_m$ is nef, applying the Noether formula \eqref{noether_formula} we have that
$$0\leq F_m\cdot \tilde{C}=mD\cdot C-\sum_{i=1}^r\text{mult}_{p_i}(\varphi_C)\nu_r(\mathfrak{m}_i)=mD\cdot C-\nu_r(\varphi_C).$$
Then, $$\dfrac{D\cdot C}{\nu_r(\varphi_C)}\geq \frac{1}{m}\text{ and, by Lemma \ref{lemm:infimo_seshadri}, }\varepsilon(D,\nu_r)>0.$$
The converse  follows by considering divisorial valuations defined by a one point blowup and applying the classical Seshadri criterion for ampleness.
\end{proof}

To conclude this subsection, we study some properties of $\varepsilon(D,\nu_r)$ and its connection with the value $\hat{\mu}_D(\nu_r)$ defined in Subsection \ref{subsec:mupico}.

\begin{proposition}\label{prop:sDnu_mupico}
Let $D$ be a big and nef divisor on $S$ and $\nu_r$ a divisorial valuation of $S$. Then, the following inequalities hold.
\begin{enumerate}
\item[(a)] $\varepsilon(D,\nu_r)\overline{\beta}_{g+1}(\nu_r)\leq \hat{\mu}_D(\nu_r)$.
\item[(b)]$\varepsilon(D,\nu_r)\hat{\mu}_D(\nu_r)\leq D^2.$
\end{enumerate}
\end{proposition}
\begin{proof}
By Remark \ref{rem:Bounds_Sdnu} and \eqref{Eq_desigualdadmupico}, one has that
$$
\varepsilon(D,\nu_r)\overline{\beta}_{g+1}(\nu_r)\leq \sqrt{\dfrac{D^2}{\overline{\beta}_{g+1}(\nu_r)}}\overline{\beta}_{g+1}(\nu_r)=\sqrt{D^2\overline{\beta}_{g+1}(\nu_r)}\leq \hat{\mu}_D(\nu_r),
$$
which proves the inequality in (a).

Let us see (b). Since $P_{\varepsilon}:=D^*-\varepsilon(D,\nu_r)\sum_{i=1}^r\nu_r(\mathfrak{m}_i)E_i^*$ is a nef divisor, in particular it holds that $P_{\varepsilon}\cdot \tilde{C}\geq 0$ for any curve $C$ on $S$ such that $C\in |mD|,$  $m\gg 0$ and $\nu_r(\varphi_C)>0$, where $\tilde{C}$ is the strict transform of $C$ on the surface $\tilde{S}$ defined by $\nu_r$. Then, the Noether formula shows
\begin{equation*}
0\leq P_{\varepsilon}\cdot \tilde{C}=mD^2-\varepsilon(D,\nu_r)\nu_r(\varphi_C),
\end{equation*}
 which implies
\begin{equation*}
D^2\geq \varepsilon(D,\nu_r)\hat{\mu}_D(\nu_r),
\end{equation*}
and the proof is completed.
\end{proof}

\subsection{Minimal plane divisorial valuations with respect to an ample divisor}\label{subsec:minimalvaluation}

This subsection studies minimality of divisorial valuations of smooth (complex) projective surfaces with respect to ample divisors. Particularizing to the projective plane and a ge\-neral projective line, we obtain several formulations of the valuative Nagata conjecture. We start by providing the Zariski decomposition of a specific family of divisors.

\begin{proposition}\label{prop:decZar}
Keep the above notation. Let $D$ be an ample divisor on $S$ and set $D_t:=D^*-tE_r,$ $0\leq t\leq \hat{\mu}_D(\nu_r)$. Then, $D_t=P_t+N_t,$ where
$$
P_t:=D^*-\frac{t}{\overline{\beta}_{g+1}(\nu_r)}\sum_{i=1}^r\nu_r(\mathfrak{m}_i)E_i^* \text{ and }N_t:=\frac{t}{\overline{\beta}_{g+1}(\nu_r)}\sum_{i=1}^{r-1}\nu_r(\varphi_i)E_i,
$$
is the Zariski decomposition of $D_t$ for $0\leq t \leq \varepsilon(D,\nu_r)\overline{\beta}_{g+1}(\nu_r).$ In addition, there is no value $t'>\varepsilon(D,\nu_r)\overline{\beta}_{g+1}(\nu_r)$ such that $P_{t'}$ is a nef \ $\mathbb{R}$-divisor.
\end{proposition}
\begin{proof}
Assume $0\leq t \leq \varepsilon(D,\nu_r)\overline{\beta}_{g+1}(\nu_r)$ and let us show that $P_t+N_t$ is the Zariski decomposition of $D_t$. The equality
$\sum_{i=1}^r\nu_r(\mathfrak{m}_i)E_i^*=\sum_{i=1}^r\nu_r(\varphi_i)E_i$ holds by \cite[Lemma 1.1.32]{Alb} and \eqref{noether_formula} and thus, $D_t=P_t+N_t$.
By (\ref{proximity_equalities}), $P_t\cdot E_i= 0,$ for $1\leq i< r,$ and $P_t\cdot E_r > 0$. To conclude the proof of our first statement, by definition of the $\mathbb{R}$-divisor $P_t,$ it suffices to prove that $P_t\cdot \tilde{C}\geq 0,$ $0<t\leq \varepsilon(D,\nu_r)\overline{\beta}_{g+1}(\nu_r),$ for all the curves $C$ on $S$ such that $\nu_r(\varphi_C)>0$. Indeed,
$$
P_t\cdot \tilde{C}= \dd\cdot\cc - \frac{t}{\overline{\beta}_{g+1}(\nu_r)}\nu_r(\varphi_C)\geq \dd\cdot\cc - \varepsilon(D,\nu_r)\nu_r(\varphi_C)\geq 0.
$$

Let us prove our second statement. Reasoning by contradiction, assume that there exists a value $t'>\varepsilon(D,\nu_r)\overline{\beta}_{g+1}(\nu_r)$ such that $P_{t'}$ is a nef $\mathbb{R}$-divisor. By Lemma \ref{lemm:infimo_seshadri}, there exists a curve $C$ on $S$ such that
\begin{equation}\label{condition:prop_decZar}
\varepsilon(D,\nu_r)\leq \frac{D\cdot C}{\nu_r(\varphi_C)}<\frac{t'}{\overline{\beta}_{g+1}(\nu_r)}.
\end{equation}
Let $\tilde{C}$ be the strict transform of $C$ on the surface $\tilde{S}$ defined by $\nu_r$. Since the $\mathbb{R}$-divisor $P_{t'}$ is nef, one has that $P_{t'}\cdot \tilde{C}\geq 0$ and then
$$
0\leq P_{t'}\cdot \tilde{C} = D^*\cdot C^* -\frac{t'}{\overline{\beta}_{g+1}(\nu_r)}\nu_r(\varphi_C)<D^*\cdot C^* -  \frac{D\cdot C}{\nu_r(\varphi_C)}\nu_r(\varphi_C)=0,
$$
which gives the desired contradiction. Notice that the second inequality holds by \eqref{condition:prop_decZar}.
\end{proof}

\begin{corollary}\label{cor:BoundZarCham}
Keep the above notation and assume that $\hat{\mu}_D(\nu_r)>\varepsilon(D,\nu_r)\overline{\beta}_{g+1}(\nu_r)$. Then, the $\mathbb{R}$-divisor $P_{\varepsilon}:=D^*-\varepsilon(D,\nu_r)\sum_{i=1}^r\nu_r(\mathfrak{m}_i)E_i^*$ on the surface $\tilde{S}$ defined by $\nu_r$ is big and nef, and the numerical equivalence class of the $\mathbb{R}$-divisor $$D_{\varepsilon}:=D^*-\varepsilon(D,\nu_r)\overline{\beta}_{g+1}(\nu_r)E_r$$ lies on the boundary of a Zariski chamber of the big cone of $\tilde{S}$.
\end{corollary}
\begin{proof}
Let us first prove that the $\mathbb{R}$-divisor $P_{\varepsilon}$ is big. $P_{\varepsilon}$ is a nef $\mathbb{R}$-divisor by Proposition \ref{prop:decZar}.  By \cite[Theorem 2.2.6]{Laz1}, $P_{s}$ is big if and only if the self-intersection of $P_{\varepsilon}$ is positive. Indeed, by Proposition \ref{prop:sDnu_mupico} and the inequality $\hat{\mu}_D(\nu_r)>\varepsilon(D,\nu_r)\overline{\beta}_{g+1}(\nu_r),$ one obtains that
$$
P_{\varepsilon}^2=(D^*)^2-\varepsilon(D,\nu_r)^2\overline{\beta}_{g+1}(\nu_r)>(D^*)^2-\hat{\mu}_D(\nu_r)\varepsilon(D,\nu_r)\geq 0,
$$
which proves our claim. Finally, our last statement follows from Proposition \ref{prop:decZar}.

\end{proof}

Next we state our first main result in this paper (Theorem \ref{Intro_thm:triangulo} in the introduction). Here, we consider a divisorial valuation of a smooth projective surface $S$ and provide several equivalent conditions to the fact that it is minimal with respect to an ample divisor on $S$. Several equivalent statements for the valuative Nagata conjecture can be immediately deduced from this result.

\begin{theorem}\label{thm:triangulo}
Let $D$ be an ample divisor on a smooth projective surface $S$, $\nu$ an exceptional curve valuation of $S$ and $\nu_r$ its first component. Keep the above notation and set $\tilde{S}$ the surface defined by the divisorial valuation $\nu_r.$ Then, the following statements are equivalent:
\begin{enumerate}[(a)]
\item The divisorial valuation $\nu_r$ is minimal with respect to $D$.

\item The Newton-Okounkov body of $D$ with respect to $\nu$ coincides with the triangle $\mathfrak{T}_D(\nu)$ defined in Proposition \ref{prop:335}.

\item The $\mathbb{R}$-divisor
$$
P_\mu:=D^* - \dfrac{\hat{\mu}_D(\nu_r)}{[\text{\emph{vol}}(\nu_r)]^{-1}} \sum_{i=1}^r \nu_r (\mathfrak{m}_i) \cdot E_i^*
$$
is nef and satisfies that $P_\mu^2=0$.

\item It holds that $\hat{\mu}_D(\nu_r)=\varepsilon(D,\nu_r)[\text{\emph{vol}}(\nu_r)]^{-1}$.

\item The segment $\{[D^*-tE_r]\mid 0\leq t\leq \hat{\mu}_D(\nu_r)\}$ crosses only one Zariski chamber of the big cone of $\tilde{S}$, $\text{ \emph{Big}}(\tilde{S}).$

\item It holds that $$\varepsilon(D,\nu_r)=\sqrt{\frac{D^2}{[\text{\emph{vol}}(\nu_r)]^{-1}}}.$$

\end{enumerate}
\end{theorem}

\begin{proof}
Recall from Subsection \ref{subsec:invariantsvaluation} that $[\text{vol}(\nu_r)]^{-1}$ equals the last maximal contact value of $\nu_r, \overline{\beta}_{g+1}(\nu_r).$ We use preferably this last notation. We begin by proving the equivalence between (a) and (b). The inclusion $\Delta_\nu(D)\subseteq \mathfrak{T}_D(\nu)$ holds by Proposition \ref{prop:335}. Then,
\begin{equation}\label{eq:ineq_triangle_thm_minval}
\dfrac{D^2}{2}=\text{Area}\left(\Delta_\nu(D)\right)\leq \text{Area}\left(\mathfrak{T}_D(\nu)\right)=\dfrac{\hat{\mu}_D(\nu_r)^2}{2\overline{\beta}_{g+1}(\nu_r)},
\end{equation}
which proves that $\nu_r$ is minimal with respect to $D$ if and only if the inequality in \eqref{eq:ineq_triangle_thm_minval} is an equality if and only if $\Delta_\nu(D)=\mathfrak{T}_D(\nu).$

Let us show that (b) implies (c). By the change of basis provided by the proximity matrix \cite[Lemma 1.1.32]{Alb} and \eqref{noether_formula}, we deduce that $D_{\mu}:=\dd-\hat{\mu}_D(\nu_r) E_r^*=P_{\mu} + N_{\mu}$, where
\begin{align*}
P_{\mu}:=& \dd - \frac{\hat{\mu}_D(\nu_r)}{\overline{\beta}_{g+1}(\nu_r)} \sum_{i=1}^{r}\nu_r(\mathfrak{m}_i) E_i^*, \text{ and}\\
N_{\mu}:=& \frac{\hat{\mu}_D(\nu_r)}{\overline{\beta}_{g+1}(\nu_r)} \sum_{i=1}^{r-1} \nu_r (\varphi_i) E_i.
\end{align*}
Let us show that $P_\mu + N_\mu$ is, in fact, the Zariski decomposition of $D_{\mu}$. It follows from the next three items: (ZD1), (ZD2) and (ZD3).

(ZD1) $N_{\mu}$ is effective by construction and the intersection matrix of its irreducible components, ${E}_1,E_2\ldots , {E}_{r-1}$, is negative definite.

(ZD2) The proximity equalities \eqref{proximity_equalities} yield  the equality $P_{\mu}\cdot {E}_i=0$ for every $i=1,2,\ldots ,$ $ r-1$. Moreover, $P_\mu \cdot E_r>0.$

(ZD3) $P_{\mu}$ is nef. Indeed, reasoning by contradiction, assume the existence of a curve $C$ on $S$ whose strict transform $\tilde{C}$ on $\tilde{S}$ satisfies that $P_{\mu}\cdot \tilde{C}<0$ (notice that, necessarily, $\nu_r(\varphi_C)>0$); this is the unique possibility because of (ZD2) and the fact that  $D$ is ample. Then
$$
0>P_{\mu}\cdot \tilde{C}=\dd\cdot \cc-\dfrac{\hat{\mu}_D(\nu_r)}{\overline{\beta}_{g+1}(\nu_r)}\nu_r(\varphi_C).
$$
Therefore, Lemma \ref{lemm:infimo_seshadri} shows that
\begin{equation}\label{eq:thmdesigmupico}
\hat{\mu}_D(\nu_r) > \frac{\dd \cdot \cc}{\nu_r(\varphi_C)} \cdot \bar{\beta}_{g+1}(\nu_r)\geq \varepsilon(D,\nu_r)\overline{\beta}_{g+1}(\nu_r).
\end{equation}
By Proposition \ref{prop:decZar}, the Zariski decomposition of the divisor
$$
D_{\varepsilon}:= \dd - \varepsilon(D,\nu_r) \bar{\beta}_{g+1}(\nu_r)E_r,
$$
is $D_{\varepsilon}=P_{\varepsilon}+N_{\varepsilon},$ where
\begin{align*}
P_{\varepsilon}:=& \, \dd - \varepsilon(D,\nu_r) \sum_{i=1}^{r}\nu_r(\mathfrak{m}_i) E_i^*,\text{ and }\\
N_{\varepsilon}:=& \, \varepsilon(D,\nu_r)\sum_{i=1}^{r-1} \nu_r (\varphi_i) {E}_i.
\end{align*}
Then, Inequality \eqref{eq:thmdesigmupico} and Corollary \ref{cor:BoundZarCham} show that the $\mathbb{R}$-divisor $D_\varepsilon$ lies on the boundary of a Zariski chamber and thus, by Theorem \ref{pro:KLM} and Remark \ref{remark:thmKLM}, $\Delta_\nu(D)$ would have a vertex different from those in $\mathfrak{T}_D(\nu),$ which is a contradiction. This proves (ZD3) and our first statement in (c).

By Proposition \ref{prop:335} and Lemma \ref{lem:336}(a), the area of $\mathfrak{T}_D(\nu)$ is $\hat{\mu}_D(\nu_r)^2/2\overline{\beta}_{g+1}(\nu_r)$. Equalling the areas of $\Delta_\nu(D)$ and $\mathfrak{T}_D(\nu),$ one gets that  $\hat{\mu}_D(\nu_r)^2=D^2\overline{\beta}_{g+1}(\nu_r)$ and then $P_\mu^2=0,$ which completes the proof of (c).

\medskip

Now we prove that (c) implies (d). Since $P_\mu$ is a nef $\mathbb{R}$-divisor, it holds that
$$
0\leq P_{\mu}\cdot \tilde{C} = \dd\cdot \cc - \frac{\hat{\mu}_D(\nu_r)}{\overline{\beta}_{g+1}(\nu_r)}\nu_r(\varphi_C),
$$
for every integral curve $C$ on $S$ such that $\nu_r(\varphi_C)>0$. Then,
$$
\hat{\mu}_D(\nu_r) \leq \frac{\dd\cdot \cc}{\nu_r(\varphi_C)}\bar{\beta}_{g+1}(\nu_r), \text{ and therefore }\hat{\mu}_D(\nu_r) = \varepsilon(D,\nu_r)\bar{\beta}_{g+1}(\nu_r),
$$
by Lemma \ref{lemm:infimo_seshadri} and Proposition \ref{prop:sDnu_mupico}.

The fact that (d) implies (e) follows from Proposition \ref{prop:decZar}.

Let us prove that (e) implies (f). We can assume that $\hat{\mu}_D(\nu_r)=\varepsilon(D,\nu_r)\overline{\beta}_{g+1}(\nu_r)$ because, otherwise, by Proposition \ref{prop:sDnu_mupico}(a) and Corollary \ref{cor:BoundZarCham}, the segment $\{[D^*-tE_r]\mid 0\leq t\leq \hat{\mu}_D(\nu_r)\}$ would cross at least two Zariski chambers of Big$(\tilde{S})$. Therefore, by Propositions \ref{pro:LM-alphabeta} and \ref{prop:decZar}, with the notation as in Subsection \ref{subsec:Nobodies}, the vertices of the Newton-Okounkov body $\Delta_\nu(D)$ of $D$ are $(0,0),(\varepsilon(D,\nu_r)\overline{\beta}_{g+1}(\nu_r), \varepsilon(D,\nu_r)\nu_r(\varphi_\eta))$ and $(\varepsilon(D,\nu_r)\overline{\beta}_{g+1}(\nu_r),$ $\varepsilon(D,\nu_r)(\nu_r(\varphi_\eta)+1))$ if $p_{r+1}$ is the satellite point $E_\eta \cap E_r$, $\eta \neq r$. Otherwise, the vertices are $(0,0),$ $(\varepsilon(D,\nu_r)\overline{\beta}_{g+1}(\nu_r), 0)$ and $(\varepsilon(D,\nu_r)\overline{\beta}_{g+1}(\nu_r),\varepsilon(D,\nu_r)).$ As a consequence, the area of $\Delta_\nu(D)$ is  $\varepsilon(D,\nu_r)^2\overline{\beta}_{g^*+1}(\nu_r)/2$. This proves that $$\varepsilon(D,\nu_r)=\sqrt{D^2/\overline{\beta}_{g+1}(\nu_r)}$$ because the area of $\Delta_\nu(D)$ is always $D^2/2.$

Only remains to show that (f) implies (a). Reasoning by contradiction, assume that the divisorial valuation $\nu_r$ is not minimal with respect to the divisor $D$. Then, from Proposition \ref{prop:sDnu_mupico}(a) it follows that $\varepsilon(D,\nu_r)\overline{\beta}_{g+1}(\nu_r)<\hat{\mu}_D(\nu_r)$. In addition, by Proposition \ref{prop:sDnu_mupico}(b) one gets
$$
D^2\geq \varepsilon(D,\nu_r)\hat{\mu}_D(\nu_r)>\overline{\beta}_{g+1}(\nu_r)\varepsilon(D,\nu_r)^2
$$
and thus $\varepsilon(D,\nu_r)<\sqrt{\frac{D^2}{\overline{\beta}_{g+1}(\nu_r)}},$ which is a contradiction with (f). This concludes the proof of the theorem.
\end{proof}

\begin{remark}\label{Rem_consThm_minimal}
Theorem \ref{thm:triangulo} proves that a divisorial valuation $\nu_r$ of a smooth projective surface is minimal with respect to an ample divisor $D$ on $S$ if and only if the Seshadri constant $\varepsilon(D,\nu_r)$ is maximal (see Definition \ref{def:seshadri_maximal}). In addition, when $\nu_r$ is minimal with respect to $D,$ the nef $\mathbb{R}$-divisor $P_{\varepsilon}=P_\mu$ lies on the boundaries of the nef and pseudoeffective cones of $\tilde{S}$. Finally, the inequality in Proposition \ref{prop:sDnu_mupico}(b) is an equality since it holds that $\hat{\mu}_D(\nu_r)=\sqrt{D^2\overline{\beta}_{g+1}(\nu_r)}$ and $\varepsilon(D,\nu_r)=\sqrt{D^2/\overline{\beta}_{g+1}(\nu_r)}.$
\end{remark}

The valuative Nagata conjecture \cite[Conjecture 4.5]{GalMonMoy} involves divisorial and irrational plane valuations and implies the original Nagata conjecture. For divisorial valuations,
it states that, if $\nu_r$ is a very general divisorial  valuation of $\mathbb{P}^2$ such that $[{\rm vol}^N(\nu)]^{-1}\geq 9$, then $\nu_r$ is minimal (with respect to a general line of $\mathbb{P}^2$). Taking, in Theorem \ref{thm:triangulo}, $S$ as the projective plane $\mathbb{P}^2$ and $D$ as a general line of $\mathbb{P}^2$, we obtain several different formulations of this conjecture. Some of them are natural valuative analogs of known reformulations of the Nagata conjecture. This fact is stated in the following corollary (Conjecture \ref{ConjectureP2} in the introduction).

\begin{corollary}\label{cor_val_Nag_conj}
The valuative Nagata conjecture for divisorial valuations admits  the following equivalent statement: if $\nu_r$ is a very general divisorial valuation of $S=\mathbb{P}^2$ such that $[{\rm vol}^N(\nu_r)]^{-1}\geq 9,$ then all the equivalent statements in Theorem \ref{thm:triangulo} hold for any exceptional curve valuation $\nu$ whose first component is $\nu_r$ and a general projective line $D$ of $\mathbb{P}^2$.
\end{corollary}

As a consequence of Corollary \ref{cor_val_Nag_conj}, to prove the valuative Nagata conjecture for divisorial valuations, it would suffice to assume that $\nu_r$ is a very general divisorial valuation of $\mathbb{P}^2$ such that $[{\rm vol}^N(\nu_r)]^{-1}\geq 9$ and prove any of the statements given in Theorem \ref{thm:triangulo}.\\

Let $S$ be a smooth projective surface and $p$ a point in $S$. Denote by $\varepsilon(D,p)$ the Seshadri constant of $D$ at $p$,   cf. \cite[Definition 5.1.1]{Laz1}. Considering $r=1$ in Theorem \ref{thm:triangulo}, Item (d) states that
\begin{equation}\label{eq:mupico_seshadri}
\hat{\mu}_D(\nu_r)=\varepsilon(D;p),
\end{equation}
and thus the remaining items are equivalent to the previous equality. The fact that Equality \eqref{eq:mupico_seshadri} holds when $\varepsilon(D;p)$ is irrational was noticed in \cite[Remark 2.1]{DumKurMacSze} and, as we will prove in our next result, $\varepsilon(D,\nu_r)$ could be irrational only when $\nu_r$ is minimal with respect to $D.$ Thus, Item (d) in Theorem \ref{thm:triangulo} is the natural extension of Remark 2.1 in \cite{DumKurMacSze}.

\medskip

Our next result (Theorem \ref{Intro_prop:submaximal_curve} in the introduction) studies submaximal curves in the valuative context (see the forthcoming Definition  \ref{def:submaximalcurves}).

\begin{theorem}\label{prop:submaximal_curve}
Let $D$ be an ample divisor on a smooth projective surface $S$ and $\nu_r$ a  divisorial valuation of $S$. If $\nu_r$ is not minimal with respect to $D$, then there exists an integral curve $C$ on $S$ with $\nu_r(\varphi_{C})>0$ such that
$$
\dfrac{D\cdot C}{\nu_r(\varphi_{C})}=\varepsilon(D,\nu_r)<\sqrt{\frac{D^2}{[\text{\emph{vol}}(\nu_r)]^{-1}}}.
$$
Moreover, if there exists an integral curve $C$ on $S$ such that $\nu_r(\varphi_{C})>0$ and $$ \dfrac{D\cdot C}{\nu_r(\varphi_C)}<\sqrt{\frac{D^2}{[\text{\emph{vol}}(\nu_r)]^{-1}}},$$ then the valuation $\nu_r$ is not minimal with respect to $D$.
\end{theorem}
\begin{proof}
Recall that $[\text{vol}(\nu_r)]^{-1}=\overline{\beta}_{g+1}(\nu_r)$. We start by showing the first equality. By Proposition \ref{prop:sDnu_mupico}, if the divisorial valuation $\nu_r$ is not minimal with respect to $D$ then $\varepsilon(D,\nu_r)\overline{\beta}_{g+1}(\nu_r)<\hat{\mu}_D(\nu_r)$. Consequently, by Corollary \ref{cor:BoundZarCham}, the segment $\{[D^*-tE_r]\mid 0\leq t\leq \hat{\mu}_D(\nu_r)\}\subseteq {\rm Num}_{\mathbb{R}}(\tilde{S})$, $\tilde{S}$ being the surface defined by $\nu_r$,  crosses at least two Zariski chambers of  Big$(\tilde{S})$. Moreover, by \cite[Proposition 1.7]{BauKurSze}, there exists a curve $C$ on $S$ such that its strict transform $\tilde{C}$ on $\tilde{S}$ satisfies
$$
0=P_{\varepsilon}\cdot \tilde{C}=\left(D^*-\varepsilon(D,\nu_r)\sum_{i=1}^r\nu_r(\mathfrak{m}_i)E_i^*\right)\cdot \tilde{C},
$$
where $P_{\varepsilon}$ is the divisor defined in Corollary \ref{cor:BoundZarCham}. As $D$ is an ample divisor on $S$, the above equalities imply that $\nu_r(\varphi_{C})>0$ and then
$$
\varepsilon(D,\nu_r)=\dfrac{D\cdot C}{\nu_r(\varphi_{C})},
$$
which  proves the first statement, since the inequality holds by Remark \ref{rem:Bounds_Sdnu} and Item (f) in Theorem \ref{thm:triangulo}.

To finish our proof, let us prove our second statement. Reasoning by contradiction, assume that $\nu_r$ is a minimal valuation with respect to $D$. By Theorem \ref{thm:triangulo}, it holds that
$$
0>D\cdot C-\nu_r(\varphi_C)\sqrt{\dfrac{D^2}{\overline{\beta}_{g+1}(\nu_r)}}=D\cdot C - \varepsilon(D,\nu_r)\nu_r(\varphi_C)=P_{\varepsilon}\cdot C,
$$
which is a contradiction since, by Proposition \ref{prop:decZar}, $P_{\varepsilon}$ is a nef $\mathbb{R}$-divisor.
\end{proof}

\begin{remark}
Assume that $\nu_r$ is not minimal with respect to $D$. By the proof of the above result, the strict transform $\tilde{C}$ of any integral curve $C$ satisfying
$$
\varepsilon(D,\nu_r)=\dfrac{D\cdot C}{\nu_r(\varphi_{C})}
$$
has negative self-intersection (because $P_{\varepsilon}\cdot \tilde{C}=0$ and $P_{\varepsilon}$ is a big and nef $\mathbb{R}$-divisor). So, its class modulo numerical equivalence generates an extremal ray of the cone ${\text{NE}}(\tilde{S})$.
\end{remark}

\begin{definition}\label{def:submaximalcurves}
Let $D$ be an ample divisor on a smooth projective surface $S$ and $\nu_r$ a divisorial valuation of $S$. An integral curve $C$ on $S$ is said to be \emph{submaximal (with respect to $D$ and $\nu_r$)} whenever $\nu_r(\varphi_C)>0$ and
$$
\dfrac{D\cdot C}{\nu_r(\varphi_C)}<\sqrt{\frac{D^2}{\overline{\beta}_{g+1}(\nu_r)}}.
$$
\end{definition}
\begin{remark}\label{remark:submaximal_onlynonminimal}
Lemma \ref{lemm:infimo_seshadri} and Theorems \ref{thm:triangulo} and \ref{prop:submaximal_curve} prove that only non-minimal divisorial valuations admit submaximal curves, since $$\varepsilon(D,\nu_r)=\sqrt{\frac{D^2}{\overline{\beta}_{g+1}(\nu_r)}}$$ if and only if $\nu_r$ is minimal with respect to $D.$
\end{remark}

When $S=\mathbb{P}^2, D=L$ is a general projective line and $\nu_r$ a non-minimal valuation, an integral curve $C$ on $\mathbb{P}^2$ such that $\nu_r(\varphi_C)=\deg(C)\hat{\mu}(\nu_r)$ is named  \emph{supraminimal} (see \cite[Section 5]{DumHarKurRoeSze} and \cite[Definition 3.11]{GalMonMoyNic2}).

Our next results consider the above context. Firstly, we prove that, in this case, suprami\-nimal and submaximal curve are the same thing. Notice that, when we consider another setting, for instance when $S$ is a Hirzebruch surface and $D$ an ample divisor on $S$, the above statement is not true; in fact, the curve which reaches $\hat{\mu}_D(\nu_r)$ can be non-integral (see \cite{GalMonMor2}).

\begin{proposition}\label{prop:submax_suprama}
Let $S=\mathbb{P}^2$ be the projective plane, $L$ a general projective line on $\mathbb{P}^2$ and $\nu_r$ a divisorial valuation of $\mathbb{P}^2$. Assume that $\nu_r$ is not minimal. An integral  curve $C$ on $\mathbb{P}^2$ is supraminimal if and only if it is submaximal. In particular, there exists a unique submaximal curve (with respect to $\nu_r$ and $L$).
\end{proposition}
\begin{proof}
If $C$ is supraminimal, then
$$
\sqrt{\dfrac{L^2}{\overline{\beta}_{g+1}(\nu_r)}}=\dfrac{1}{\sqrt{\overline{\beta}_{g+1}(\nu_r)}}>\dfrac{\deg(C)}{\nu_r(\varphi_C)}=\dfrac{L\cdot C}{\nu_r(\varphi_C)},
$$
and therefore $C$ is submaximal.

Conversely, if $C$ is a submaximal curve, then $\nu_r(\varphi_C)>\sqrt{\overline{\beta}_{g+1}(\nu_r)}\deg(C)$ and, by \cite[Lemma 3.10]{GalMonMoyNic2}, $C$ is supraminimal.

The last statement in the proposition follows from Theorem \ref{prop:submaximal_curve} and the uniqueness of the supraminimal curve \cite{DumHarKurRoeSze,GalMonMoyNic2}.
\end{proof}

\begin{corollary}\label{cor:epsilon_propertyP2}
Let $\nu_r$ be a divisorial valuation of $\mathbb{P}^2$ and $L$ a general projective line on $\mathbb{P}^2$. Denote by $\nu_r^N$ the normalization of $\nu_r$. Then the following conditions hold:
\begin{itemize}
\item[(a)] The valuative Nagata conjecture for divisorial valuations of $\mathbb{P}^2$ can be equivalently stated as: if $\nu_r$ is a very general divisorial valuation of $\mathbb{P}^2$ such that $[{\rm vol}^N(\nu_r)]^{-1}\geq 9,$ then there is no submaximal curve on $\mathbb{P}^2$ with respect to $L$ and $\nu_r$.
\item[(b)] It holds that $$ \hat{\mu}_L(\nu_r)\,\varepsilon(L,\nu_r)=1. $$
\item[(c)] Let $\Gamma$ be the dual graph of a divisorial valuation. Then, the largest value in the set
$$\Big\{\varepsilon(L,\nu_r^N) \ | \ \Gamma \text{ is the dual graph of }\nu_r \Big\}$$
is reached for very general valuations $\nu_r.$
\item[(d)] Set $t(\nu_r)=1$ and $\delta_0(\nu_r)=-1$ if $r=1$. Otherwise, $t(\nu_r):=\nu_r(\varphi_H),$ $H$ being the projective line such that $\nu_r(\varphi_H)>\overline{\beta}_0,$ and define
$$
\delta_0(\nu_r):=\Bigg\lceil\dfrac{\overline{\beta}_{g+1}(\nu_r)-2\overline{\beta}_0(\nu_r)t(\nu_r)}{t(\nu_r)^2}\Bigg\rceil^+,
$$
where $\lceil x\rceil^+$ is defined as the ceiling of a rational number $x$ if $x\geq 0$, and $0$ otherwise. Then,
$$
\varepsilon(L,\nu_r)\geq \dfrac{1}{\overline{\beta}_0(\nu_r)+ t(\nu_r)(1+\delta_0(\nu_r))}.
$$
\end{itemize}
\end{corollary}
\begin{proof}
Item (a) follows from Remark \ref{remark:submaximal_onlynonminimal}. Let us prove Item (b). Firstly assume that $\nu_r$ is a minimal valuation with respect to $L$. By Theorem \ref{thm:triangulo}, one has that
$$
\hat{\mu}_L(\nu_r)=\sqrt{\overline{\beta}_{g+1}(\nu_r)} \text{ and } \varepsilon(L,\nu_r)=\sqrt{1/\overline{\beta}_{g+1}(\nu_r)}
$$
and then $\hat{\mu}_L(\nu_r)\,\varepsilon(L,\nu_r)=1.$ Otherwise, by Proposition \ref{prop:submax_suprama}, it holds that
$$
\hat{\mu}_L(\nu_r)=\frac{1}{\varepsilon(L,\nu_r)},
$$
which proves Item (b). Finally, Item (c) (respectively, Item (d)) follows from Item (b) and \cite[Corollary 4.3]{GalMonMoy} (respectively, Item (b) and \cite[Corollary 3.5]{GalMonMorPer}).
\end{proof}

We conclude this article by adding some information about submaximal curves on smooth projective surfaces $S$ and also about  infinitesimal Newton-Okounkov bodies of ample divisors on $S$ with respect to exceptional curve valuations.

The following result can be regarded as an infinitesimal version of \cite[Proposition 2.7]{StrSze} and its proof.

\begin{proposition}\label{prop:numsubmaximalcurves}
Let $D$ be an ample divisor on $S$ and $\nu_r$ a divisorial valuation of $S$. Set $\rho(S)$ the Picard number of $S$ and $\tilde{S}$ the surface defined by $\nu_r$.  If $\nu_r$ is not minimal with respect to $D$, then there exist at most $\rho(S)$ submaximal curves on $S$ computing $\varepsilon(D,\nu_r)$.
\end{proposition}
\begin{proof}
By Corollary \ref{cor:BoundZarCham}, the $\mathbb{R}$-divisor $P_{\varepsilon}:=D^*-\varepsilon(D,\nu_r)\sum_{i=1}^r\nu_r(\mathfrak{m}_i)E_i^*$ on $\tilde{S}$ is big and nef. Consider the set $\{C_i\}_{i=1}^n$ of  submaximal curves on $S$ computing $\varepsilon(D,\nu_r)$. Set $\tilde{C}_i$ the strict transform of $C_i$ on $\tilde{S}$. Pick $\alpha_i\geq  0$, and then,
$$
P_{\varepsilon}\cdot \sum_{i=1}^n \alpha_i\tilde{C}_i=\sum_{i=1}^n \alpha_i \left(P_{\varepsilon}\cdot\tilde{C}_i\right)=0,
$$
where the last equality holds by Remark \ref{remark:submaximal_onlynonminimal}. Therefore, by \cite[Lemma 4.3]{BauKurSze} the intersection matrix given by $\tilde{C}_1,\tilde{C}_2,\ldots,\tilde{C}_n$ is negative definite and, by Hodge index Theorem \cite[Chapter V, Theorem 1.9]{Har}, $n\leq \rho(S) +r -1$. Moreover, by Proposition 3.9, the negative part of the Zariski decomposition of  $D^*-\varepsilon(D,\nu_r)\overline{\beta}_{g+1}(\nu_r)E_r,$ $$N_\varepsilon=\varepsilon(D,\nu_r)\sum_{i=1}^{r-1}\nu_r(\varphi_i)E_i,$$ has $r-1$ components and, by Hodge index Theorem again,  the number of irreducible components of the negative part of $D^*-tE_r,t>\varepsilon(D,\nu_r)\overline{\beta}_{g+1}(\nu_r),$  is at most $\rho(S)+r-1$. Consequently, $n\leq \rho(S).$
\end{proof}

The next corollary follows from the proof of Proposition \ref{prop:numsubmaximalcurves} and the fact that the submaximal curves computing $\varepsilon(D,\nu_r)$ appear in the negative part of the Zariski decomposition of the $\mathbb{R}$-divisors $D^*-tE_r, t>\varepsilon(D,\nu_r)\overline{\beta}_{g+1}(\nu_r).$

\begin{corollary}\label{cor:boundnumberZarCha_submaximalcurves}
Let $D$ be an ample divisor on $S$ and $\nu_r$ a divisorial valuation of $S$. Assume that $\nu_r$ is not minimal with respect to $D.$ Set $n$ the number of submaximal curves on $S$ computing $\varepsilon(D,\nu_r)$ and $\rho(S)$ the Picard number of $S$. Then the segment $\{[D^*-tE_r]\mid 0\leq t\leq \hat{\mu}_D(\nu_r)\}$ crosses, at most, $2+\rho(S)-n$ Zariski chambers of the big cone of the surface $\tilde{S}$ given by $\nu_r$.
\end{corollary}

Let $\nu$ be an exceptional curve valuation of a smooth projective surface $S$ and $D$ an ample divisor on $S$. To determine the number of vertices of the Newton-Okounkov body $\Delta_\nu(D)$ is an open problem and several conjectures have been proposed in \cite{MoyFerNicRoe}. We have just proved that, when $\nu$ is not minimal with respect to $D,$ $\Delta_\nu(D)$ could have three or more vertices and the number of vertices of $\Delta_\nu(D)$ depends on the submaximal curves on $S$ which compute $\varepsilon(D,\nu_r)$. In addition, the fact of having only a submaximal curve computing $\varepsilon(D,\nu_r)$ is a necessary condition for $\Delta_\nu(D)$ to have $2\rho(S) +2$ vertices.

Our last result determines a triangle which is always included in $\Delta_\nu(D).$

\definecolor{aqaqaq}{rgb}{0.6274509803921569,0.6274509803921569,0.6274509803921569}
\definecolor{eqeqeq}{rgb}{0.8784313725490196,0.8784313725490196,0.8784313725490196}
\definecolor{uuuuuu}{rgb}{0.26666666666666666,0.26666666666666666,0.26666666666666666}
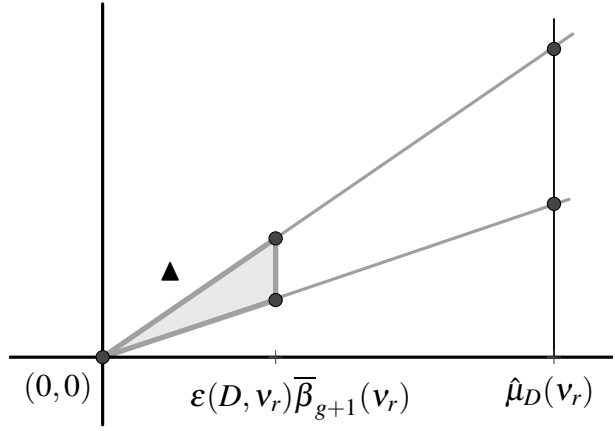
\begin{figure}[ht!]
\begin{center}
\begin{tikzpicture}[line cap=round,line join=round,>=triangle 45,x=1.0cm,y=1.0cm]
\clip(-1.2364583977696633,-1.0882371964892057) rectangle (6.773570558375912,4.857050778397071);
\fill[line width=2.pt,color=eqeqeq,fill=eqeqeq,fill opacity=0.699999988079071] (0.,0.) -- (2.3,0.76) -- (2.3,1.58) -- cycle;
%\fill[line width=2.pt,color=aqaqaq,fill=aqaqaq,fill opacity=0.550000011920929] (0.,0.) -- (2.3,1.5333333333333332) -- (4.4,2.9333333333333336) -- (6.,4.) -- (4.4,1.7862948958018658) -- (2.3,0.7666666666666666) -- cycle;
%\draw [line width=2.pt,color=eqeqeq] (0.,0.)-- (6.,4.);
\draw [line width=0.7pt] (6.,4.5)-- (6.,0);
%\draw [line width=2.pt,color=eqeqeq] (6.,2)-- (0.,0.);
\draw [line width=1.2pt,color=aqaqaq] (0.,0.)-- (6.25,4.3);
\draw [line width=1.2pt,color=aqaqaq] (0.,0.)-- (6.2,2.1);
%\draw [line width=1.2pt,color=aqaqaq] (2.3,1.5333333333333332)-- (4.4,2.9333333333333336);
%\draw [line width=1.2pt,color=aqaqaq] (4.4,2.9333333333333336)-- (6.,4.);
%\draw [line width=1.2pt,color=aqaqaq] (6.,4.)-- (4.4,1.7862948958018658);
%\draw [line width=1.2pt,color=aqaqaq] (4.4,1.7862948958018658)-- (2.3,0.7666666666666666);
%\draw [line width=2.pt,color=aqaqaq] (2.3,0.7666666666666666)-- (0.,0.);
\draw [line width=2.pt,color=aqaqaq] (2.3,0.76)-- (2.3,1.58);
\draw [line width=2.pt,color=aqaqaq] (2.3,0.76)-- (0,0);
\draw [line width=2.pt,color=aqaqaq] (2.3,1.58)-- (0,0);
\draw [line width=1.2pt] (0.,-0.9) -- (0.,4.7);
\draw [line width=1.2pt,domain=-1.2364583977696633:6.773570558375912] plot(\x,{(-0.-0.*\x)/1.});
\draw (-1.2,-0.04854397920613749) node[anchor=north west] {$(0,0)$};
\draw (1.01254602779174,-0.10711824496856387) node[anchor=north west] {$\varepsilon(D,\nu_r)\overline{\beta}_{g+1}(\nu_r)$};
%\draw (4.079154968008845,-0.12176181140917046) node[anchor=north west] {$t_2$};
\draw (5.2,-0.10711824496856387) node[anchor=north west] {$\hat{\mu}_D(\nu_r)$};
%\draw (1.9558383342956667,2.265139518409704) node[anchor=north west] {$Q_6$};
\draw (0.6,1.4) node[anchor=north west] {$\blacktriangle$};
%\draw (3.976650027070967,3.7294961624703635) node[anchor=north west] {$Q_8$};
%\draw (5.4,4.65) node[anchor=north west] {$Q$};
%\draw (6,2.72) node[anchor=north west] {$P$};
%\draw (4.064511405017719,1.4304562312951283) node[anchor=north west] {$Q_7$};
%\draw (2.014412586260168,0.6836343428241921) node[anchor=north west] {$Q_5$};
\begin{scriptsize}
\draw [fill=uuuuuu] (0.,0.) circle (2.5pt);
\draw [fill=uuuuuu] (6.,4.1) circle (2.5pt);
\draw [fill=uuuuuu] (6.,2.04) circle (2.5pt);
\draw [fill=uuuuuu] (2.3,0.76) circle (2.5pt);
\draw [fill=uuuuuu] (2.3,1.58) circle (2.5pt);
%\draw [fill=uuuuuu] (4.4,2.9333333333333336) circle (2.5pt);
%\draw [fill=uuuuuu] (4.4,1.7862948958018658) circle (2.5pt);
\draw [color=uuuuuu] (2.3,0.)-- ++(-2.5pt,0 pt) -- ++(5.0pt,0 pt) ++(-2.5pt,-2.5pt) -- ++(0 pt,5.0pt);
%\draw [color=uuuuuu] (4.4,0.)-- ++(-2.5pt,0 pt) -- ++(5.0pt,0 pt) ++(-2.5pt,-2.5pt) -- ++(0 pt,5.0pt);
\draw [color=uuuuuu] (6.,0.)-- ++(-2.5pt,0 pt) -- ++(5.0pt,0 pt) ++(-2.5pt,-2.5pt) -- ++(0 pt,5.0pt);
\end{scriptsize}
\end{tikzpicture}
\caption{The triangle $\blacktriangle$ in Proposition \ref{pro:triangle}.}\label{fig:subtriangle}
\end{center}
\end{figure}

\begin{proposition}\label{pro:triangle}
Let $D$ be an ample divisor on a smooth projective surface $S$.
Let $\nu$ be an exceptional curve valuation of $S$ defining a flag $E_\bullet$ as in \eqref{eq:flag}. Denote by $\nu_r$ the  first component of $\nu$. Set $\pi$ the finite simple sequence of blowups defined by $\nu_r$ (see \eqref{Eq_sequencepointblowingups_divvaluation}) and keep the notation as in Section \ref{sec:pre}. Then, the Newton-Okounkov body of $D$ with respect to $\nu,\Delta_\nu (D),$ contains the triangle $\blacktriangle$ whose vertices are $\textbf{0}=(0,0)$ and the following two points $P_1\text{ and }P_2$:
\begin{enumerate}
\item[(a)] When $q=p_{r+1}$ is the satellite point $E_r\cap E_\eta,\eta\neq r$, then
$$
P_1=\varepsilon(D,\nu_r)\left(\overline{\beta}_{g+1}(\nu_r), \nu_r(\varphi_\eta)\right) \text{ and } P_2=\varepsilon(D,\nu_r)\left(\overline{\beta}_{g+1}(\nu_r), \nu_r(\varphi_\eta)+1\right),
$$
where $\varphi_\eta$ is as defined in Lemma \ref{lem:336}.
\item[(b)] Otherwise (i.e., when $p_{r+1}$ is free),
$$
P_1=\varepsilon(D,\nu_r)\left(\overline{\beta}_{g+1}(\nu_r), 0\right) \text{ and } P_2=\varepsilon(D,\nu_r)\left(\overline{\beta}_{g+1}(\nu_r), 1\right).
$$
\end{enumerate}

\end{proposition}
\begin{proof}
The result follows from Theorem \ref{thm:triangulo}, Lemma \ref{lem:336} and Proposition \ref{prop:335} when the exceptional curve valuation is minimal with respect to $D$. Otherwise, it can be deduced from Theorem \ref{pro:LM-alphabeta} and Proposition \ref{prop:decZar}.  See Figure \ref{fig:subtriangle}.
\end{proof}

\begin{remark}
The triangle $\blacktriangle$ described in Proposition \ref{pro:triangle} is the largest simplex having the origin as a vertex and contained in $\Delta_\nu (D).$  In addition, if we assume that the divisorial valuation $\nu_r$ is defined by a point blowup, $\blacktriangle$ is just the inverted standard simplex described in \cite{KurLoz1} or \cite{ParkShin} for the case of surfaces.
\end{remark}

\section*{Acknowledgements}
The third author would like to thank the Department of Algebra, Analysis, Geometry and Topology, and IMUVa of Valladolid University for the support received when preparing this article.

%%%%%%%%%%%%%%%%%%%%%%%%%%%%%%%%%%%%%%%%%%%%%%%%%%%%%%%%%%%%%%%%%%%
%%%%%%%%%%%%%%%%%%%%%%%%%%%% Bibliografía %%%%%%%%%%%%%%%%%%%%%%%%%
%%%%%%%%%%%%%%%%%%%%%%%%%%%%%%%%%%%%%%%%%%%%%%%%%%%%%%%%%%%%%%%%%%%
\bibliographystyle{plain}
\bibliography{BIBLIO_paquete1}

\end{document}